 \documentclass[10pt,twocolumn,twoside,dvipsnames]{IEEEtran}
\usepackage{cite}

\ifCLASSINFOpdf
  \usepackage[pdftex]{graphicx}
\else
\fi
\usepackage{amsmath}
\usepackage{chemfig}
\usepackage{amsfonts}

\usepackage{amsthm}
\usepackage[shortlabels]{enumitem}

\newtheorem{theorem}{Theorem}
\newtheorem{prop}{Proposition}
\newtheorem{lemma}{Lemma}
\newtheorem{corollary}{Corollary}
\DeclareMathOperator*{\argmax}{argmax}
\DeclareMathOperator*{\argmin}{argmin}
\newtheorem{definition}{Definition}[section]
\theoremstyle{remark}
\newtheorem{remark}{Remark}

\usepackage{xcolor}

\begin{document}
%
\title{Adaptive Susceptibility and Heterogeneity \\ in  Contagion Models on Networks}
%
%
%

\author{Renato~Pagliara and~Naomi~Ehrich~Leonard,~\IEEEmembership{Fellow,~IEEE}
\thanks{The authors are with the Department of Mechanical and Aerospace Engineering, Princeton University, Princeton, NJ, USA {\tt\small \{renato,naomi\}@princeton.edu}}
\thanks{The research was supported in part by Army Research Office grant W911NF-18-1-0325 and Office of Naval Research grant N00014-19-1-2556.}
}

\maketitle

\begin{abstract}
Contagious processes, such as spread of infectious diseases, social behaviors, or computer viruses, affect biological, social, and technological systems. Epidemic models for large populations and  finite populations on networks have been used to understand and control both transient and steady-state  behaviors. Typically it is assumed that after recovery from an infection, every agent will either return to its original susceptible state or acquire full immunity to reinfection.
We study the network SIRI (Susceptible-Infected-Recovered-Infected) model, an epidemic model for the spread of contagious processes on a network of heterogeneous agents that can adapt their susceptibility to reinfection. The model generalizes existing models to accommodate realistic conditions in which agents acquire partial or compromised immunity after first exposure to an infection. We prove necessary and sufficient conditions on model parameters and network structure that distinguish four dynamic regimes: infection-free, epidemic, endemic, and bistable.  For the bistable regime, which is not accounted for in traditional models, we show how there can be a rapid resurgent epidemic after what looks like convergence to an infection-free population.
We use the model and its predictive capability to show how control strategies can be designed to mitigate  problematic contagious behaviors.
\end{abstract}
%
\IEEEpeerreviewmaketitle

\section{Introduction}
%
%
%
%

\IEEEPARstart{C}{ontagious} processes affect biological, social, and technological network systems. In biology, a concern is the spread of disease 
across a population~\cite{anderson1992infectious}.  In social networks, the spread of information, behaviors, or cultural norms has important implications for decisions  about politics, the environment, health care, etc.  In many contexts a spike in the ``infected'' population is detrimental, e.g., in the spread of misinformation \cite{roozenbeek2019}. In other cases, the spike can be crucial, e.g., in the spread of safety instructions in an emergency \cite{carlson2014}. In technological networks, like computer networks and mobile sensor networks, the spread of a virus can lead to disruptions, while the spread of information on a changing environment can be critical to a successful mission. To understand and control the dynamics of contagion, models with sufficiently good predictive capability are warranted.


Epidemic models have been successfully used to study contagious processes in a large number of systems, ranging from the spread of infectious diseases on populations~\cite{lajmanovich1976,Kuniya2019} and memes on social networks~\cite{JinTwitter} to the evolution of riots~\cite{bonnasse2018epidemiological} and power grid failures~\cite{BernsteinPowerGrids}. The wide applicability of epidemic models has led to an increase in recent years in the number of studies in the physics and control communities focusing on theoretical epidemic models for the propagation of contagious processes on networks~\cite{pastor2015epidemic,nowzari2016analysis,PareTimeVarying,YangBiVirus,mei2017dynamics}. Studies are typically based on the SIS (Susceptible-Infected-Susceptible)  model~\cite{fall2007epidemiological,van2009virus,van2015accuracy,PareTimeVarying,YangBiVirus,mei2017dynamics}, in which every recovered individual experiences no change to its susceptibility to the infection after recovery, or on the SIR (Susceptible-Infected-Recovered) model~\cite{youssef2011individual,mei2017dynamics}, in which recovered individuals gain full immunity to the infection. {\color{black} The SIS model captures endemic behaviors in which there is a constant fraction of the population that is infected at steady-state, while the SIR model captures epidemic behaviors in which there is a rapid increase in the number of infected individuals that eventually subsides.}

Although the SIR and SIS models are useful, they fall short in accounting for what can happen in many real-world systems when agents adapt their susceptibility in more general ways after their first exposure to the infection.  For example, in the case of infectious diseases, the susceptibility of individuals to the infection can decrease after a first exposure, resulting in partial immunity as in the case of influenza~\cite{clements1986serum}. Or susceptibility can increase, resulting in compromised immunity as in the case of  tuberculosis in particular populations~\cite{verver2005rate}.
In the spread of social behaviors, the susceptibility of individuals to the ``infection'' might decrease (increase) as a result of a negative (positive) past experience that decreases (increases) the propensity of an individual to engage in the behavior. 

For example, in psychology there is a well-established ``inoculation theory''~\cite{mcguire1961,compton2019,roozenbeek2019},  which has shown that an effective way to combat the spread of misinformation is to pre-expose people to misinformation so that they develop at least partial cognitive immunity to subsequent contact with misinformation. 
Examples also abound in social animal behavior. For example, desert harvester ants regulate foraging for seeds by means of a contagious process in which successful ants returning to the nest motivate available ants in the nest to go out and forage \cite{pinter2013harvester}. Resilience of foraging rates to temperature and humidity changes outside the nest has been attributed to adaptation of susceptibility to the ``infection'' by available ants that have already been exposed to outside conditions~\cite{pagliaraAnts}.

In technological systems, such as systems comprised of interconnected mechanical parts, susceptibility to cascading failures can increase the second time around when parts become worn or compromised the first time.
When there is the opportunity for learning and design, such as in a network of autonomous robots, protocols can be designed so that agents modulate their response to an ``infected'' agent based on what they have learned from previous interactions.

In the present paper, we analyze the role of adaptive susceptibility  in the spread of a contagious process over a network of heterogeneous agents using the network SIRI (Susceptible-Infected-Recovered-Infected) epidemic model.  The network SIRI model describes susceptible agents that become infected from contact with infected agents, infected agents that  recover, and recovered agents that become reinfected from contact with infected agents. Susceptibility is adaptive when the rates of infection and reinfection differ.  
An agent acquires {\em partial immunity} ({\em compromised immunity}) to its neighbor, if after recovering from a first-time infection, it experiences {\color{black}a decrease (increase)} in susceptibility to that neighbor.  
Every agent may have a different rate of recovery and different rates of susceptibility to infection and reinfection to each of its neighbors.
SIS and SIR  are special cases of SIRI.


To the best of our knowledge we provide the first rigorous and comprehensive analysis of the network SIRI model. 
Models that consider reinfection usually only consider the case of partial immunity~\cite{gomes2004infection,stollenwerk2007phase,stollenwerk2010spatially}. 
In~\cite{pastor2015epidemic} the authors studied a discrete-time mean-field model 
with global recovery, infection, and reinfection rates, and showed through numerical simulations that when the reinfection rate is larger than the infection rate, there is an abrupt transition from an infection-free to an endemic steady state.
We formalize this observation by proving new results on the existence of a bistable regime in which a critical manifold of initial conditions separates solutions for which the infection dies out from solutions for which the infection spreads. 

Our analysis of the network SIRI model generalizes our novel results on the SIRI model of a well-mixed population~\cite{pagliaraSIRIWM}. Our contributions are as follows. First, we introduce the network SIRI model over strongly connected digraphs and present a rigorous stability analysis. We show  there can exist only a set of non-isolated infection-free equilibria (IFE) and a stable isolated endemic equilibrium (EE), and we prove conditions on the graph structure and system parameters that determine the stability of the IFE. 
Second, we prove that the model exhibits the same four distinct behavioral regimes observed in the well-mixed SIRI model: infection-free, epidemic, endemic, and bistable. We show how the four behavioral regimes are characterized by four numbers that generalize, to the network setting, the two reproduction numbers  of~\cite{pagliaraSIRIWM} and generalize previous results for the network SIS and SIR models in~\cite{fall2007epidemiological,van2009virus,van2013inhomogeneous,youssef2011individual,mei2017dynamics}. Third, we prove  features of the geometry of solutions near the IFE  in the epidemic and bistable regimes that dictate both transient and steady-state possibilities. 
For the bistable regime, which is not accounted for by the SIS and SIR models, we show how solutions can exhibit a resurgent epidemic in which an initial infection appears to die out for an arbitrarily long period of time before abruptly resurging to an endemic steady state. Finally, we show how our results can be used to design control strategies that guarantee the eradication or spread of the infection through a network of heterogeneous agents with adaptive susceptibility.

In Section~\ref{sec:Math} we present  notation and well-known results used in the paper. In Section~\ref{sec:Model} we introduce the network SIRI model and its classification into six cases.
In Section~\ref{sec:EqlAndRepNum} we study the equilibria and define new notions of reproduction numbers. In Section~\ref{sec:Stability} we analyze stability, and in Section~\ref{sec:DynamicRegimes} we prove our main result on conditions for the four behavioral regimes. In Section~\ref{sec:BisAndEpi} we examine the geometry and behavior of solutions in the epidemic and bistable regimes.
We apply our theory to control laws that guarantee desired steady-state behavior in Section~\ref{sec:Control}. We make final remarks in Section~\ref{sec:Conclusion}.

\section{Mathematical Preliminaries}\label{sec:Math}

\subsection{Notation}
We denote the $j$-th entry of $x \in \mathbb{R}^{N}$ as $x_j$, and the $(j,k)$-th entry of $M \in \mathbb{R}^{N \times N}$ as $m_{jk}$. We define $e^j \in \mathbb{R}^N$, $j = 1,\dots,N$ as the standard basis vectors. 
We define $\mathbf{0} \in \mathbb{R}^{N}$  as the zero vector, $\mathbf{1} \in \mathbb{R}^{N}$  as the vector with every entry 1, $\bar{\mathbf{0}} \in \mathbb{R}^{N \times N}$ as the zero square matrix, and $\mathbb{I} \in \mathbb{R}^{N \times N}$ as the identity matrix. We let $\textrm{diag}(x) \in \mathbb{R}^{N \times N}$  be the diagonal matrix with entries given by the entries of $x \in \mathbb{R}^N$. 

For any vectors $x,y \in \mathbb{R}^N$, we write $x \gg y$ if $x_j > y_j$ for all $j$, $x \succ y$ if $x_j \geq y_j$ for all $j$, but $x \neq y$, and $x \succeq y$ if $x_j\geq y_j$ for all $j$. Similarly, for any two matrices $M,Q \in \mathbb{R}^{N \times N}$ we write $M \gg Q$ if $m_{jk}>q_{jk}$, $M \succ Q$ if $m_{jk} \geq q_{jk}$, for any $j,k$, but $M \neq Q$, and $M \succeq Q$ if $m_{jk}\geq q_{jk}$ for any $j,k$. 

A square matrix $M$ is Hurwitz (stable) if it has no eigenvalue with positive or zero real part, and it is unstable if at least one of its eigenvalues has positive real part. 
A real square matrix $M$ is Metzler if $m_{jk}\geq 0$ for $j \neq k$. 
We denote the spectrum of a square matrix $M$ as $\lambda(M) = \{\lambda_1,\lambda_2,\dots,\lambda_N\}$, its spectral radius as $\rho(M) = \max{\{ |\lambda_j| \, \big| \,  \lambda_j \in \lambda(M) \}}$ {\color{black}and its leading eigenvalue as  $\lambda_{max}(M) = \argmax_{\lambda_j \in \lambda(M)} |\lambda_j|$}.

A weighted digraph $\mathcal{G = (V,E)}$ consists of a set of nodes $\mathcal{V}$ and a set of edges $\mathcal{E} \subseteq \mathcal{V} \times \mathcal{V}$. Each edge $(j,k) \in \mathcal{E}$ from node $j \in \mathcal{V}$ to node $k \in \mathcal{V}$ has an associated weight $a_{jk} >0$ with $a_{jk} = 0$ if $(j,k) \not\in \mathcal{E}$.
The set of neighbors of node $j$ is $\mathcal{N}_j = \{k \in \mathcal{V}| (j,k) \in \mathcal{E}\}$. $\mathcal{G}$ is {\em strongly connected} if there exists a directed path from any node $j \in \mathcal{V}$ to any other node $k \in \mathcal{V}$. The adjacency matrix $A = \{a_{jk}\}$ of  $\mathcal{G}$ is {\em irreducible} if $\mathcal{G}$ is strongly connected. The degree of node $j$ is $d_j = \sum_{k=1}^N a_{jk}$.   $\mathcal{G}$ is {\em $d$-regular} if $d_j=d$ for all $j=1, \ldots, N$.

\subsection{Properties of Metzler Matrices}
We use well-known properties of Metzler matrices, which we summarize in the following three propositions (see~\cite{BermanPlemmons1994} Theorem 6.2.3, \cite{farina2011positive} Theorem 11 and 17, and~\cite{meyer2000matrix} Ch. 8).

\begin{prop}\label{prop:Metzler} Let $K$ be a Metzler matrix.  Then,
\begin{enumerate}
\item $\lambda_{max}(K) \in \mathbb{R}$.
If $K$ is irreducible, $\lambda_{max}(K)$ has multiplicity one.
\item Let $w^T$ and $v$ be left and right eigenvectors corresponding to $\lambda_{max}(K)$. Then, $w, v \succeq \mathbf{0}$. If $K$ is irreducible, then $w, v \gg \mathbf{0}$, and every other eigenvector of $K$ has at least one negative entry.
\item Let $K_{min},K_{max}$ be irreducible Metzler matrices where \\ $K_{min} \prec K \prec K_{max}$, then 
\begin{equation*}
\lambda_{max}(K_{min}) < \lambda_{max}(K) < \lambda_{max}(K_{max}).
\end{equation*}
\end{enumerate}
\end{prop}

\begin{prop}\label{prop:MetzlerHurwitz} Let $K$ be a Metzler matrix.  Then, the following statements are equivalent:
\begin{enumerate}
\item K is Hurwitz.
\item There exists a vector $v \gg \mathbf{0}$ such that $Kv \ll \mathbf{0}$.
\item There exists a vector $w \gg \mathbf{0}$ such that $w^T K \ll \mathbf{0}$.
\end{enumerate}
\end{prop}

\begin{definition}[Regular Splitting]\label{def:regsplitting}
Let $K$ be a Metzler matrix. $K = T + U$ is a {\em regular splitting} of $K$ if $T\succeq \bar{\mathbf{0}}$ and $U$ is a Hurwitz Metzler matrix.

\begin{prop}~\label{prop:MetzlerRS}
Let $K$ be a Metzler matrix and let $K=T+U$ be a regular splitting. Then, 
\begin{enumerate}
\item $\lambda_{max}(K)<0$ if and only if $\rho(-TU^{-1})<1$.
\item $\lambda_{max}(K)=0$ if and only if $\rho(-TU^{-1})=1$.
\item $\lambda_{max}(K)>0$ if and only if $\rho(-TU^{-1})>1$.
\end{enumerate}
\end{prop}
\end{definition}

\subsection{Properties of Gradient Systems}
A gradient system on an open set $\Omega \subseteq \mathbb{R}^N$ is a system of the form $\dot{\zeta} = - \nabla V(\zeta)$ where $\zeta(t) \in \Omega$,  $V \in C^2(\Omega)$ is the potential function, and $\nabla V = [\partial V/\partial \zeta_1,\dots,\partial V/\partial \zeta_N]$ is the gradient of $V$ with respect to $\zeta$. The {\em level surfaces} of $V$ are the subsets $V_c = \{V^{-1}(c)\in \Omega  \, | \,c \in \mathbb{R}\}$. A point $\zeta_0 \in \Omega$ is a {\em regular point} if $\nabla V(\zeta_0) \neq \mathbf{0}$ and a {\em critical point} if $\nabla V(\zeta_0) = \mathbf{0}$. If $\nabla V(\zeta) \neq \mathbf{0}$ for all $\zeta \in V_c$, then $c$ is a {\em regular value} for $V$.
\begin{prop}[Properties of Gradient Systems~\cite{hirsch2012differential,tu2012dynamical}]\label{prop:gradSys}
Consider the gradient system $\dot{\zeta} = - \nabla V(\zeta)$ where $V \in C^2(\Omega)$, $\zeta(t) \in \Omega \subseteq \mathbb{R}^N$. 
Then,
\begin{enumerate}
    \item $V(\zeta)$ is a Lyapunov function of the gradient system. Moreover, $\dot{V}(\zeta)=0$ if and only if $\zeta$ is an equilibrium.
    \item The critical points of $V$ are the system equilibria. 
    \item If $c$ is a regular value for $V$, then the surface set $V_c$ forms an $N-1$ dimensional surface in $\Omega$ and the vector field is perpendicular to $V_c$.
    \item At every point $\zeta \in \Omega$, the directional derivative along $w \in \mathbb{R}^N$ is given by $D_w V(\zeta) = w^T \nabla V(\zeta)$.
    \item Let $\zeta_0$ be an $\alpha$-limit point or an $\omega$-limit point of a solution of the gradient system. Then $\zeta_0$ is an equilibrium.
    \item The linearized system at any equilibrium has only real eigenvalues. No periodic solutions are possible.
\end{enumerate}
\end{prop}


\section{Network SIRI Model Dynamics}\label{sec:Model}
In this section we present the network SIRI model dynamics, which represent a contagious process with reinfection in a population of $N$ agents. Consider a strongly connected digraph $\mathcal{G = (V,E)}$ with adjacency matrix $A$, where each node in $\mathcal{V}$ represents an agent. The state of each agent $j$ is given by the random variable $X_j(t) \in \{S,I,R\}$, where $S$ is ``susceptible'', $I$ is ``infected'', and $R$ is ``recovered''. Let transitions between states for each agent be independent Poisson processes with rates defined as follows. Susceptible agent $j$ becomes infected through contact with infected neighbor $k$ at the rate {\color{black} $\beta_{jk} \geq 0$. We assume $\beta_{jk} =0$ if and only if $a_{jk}=0$. Infected agent $j$ recovers from the infection at the rate $\delta_j \geq 0$. Recovered agent $j$ becomes reinfected through contact with infected neighbor $k$ at the rate $\hat{\beta}_{jk} \geq 0$, where $\hat{\beta}_{jk} = 0$ if $a_{jk} = 0$.}
These transitions are summarized as
\begin{center}
\schemestart $S_j+I_k$\arrow{->[$\beta_{jk}$]}$I_j+I_k$\arrow{<-[$\hat{\beta}_{jk}$]}$R_j+I_k$\schemestop \par
\end{center}
\begin{center}
\schemestart $I_j$\arrow{->[$\delta_j$]}$R_j$.\schemestop \par
\end{center}
The dynamics are described by a continuous-time Markov chain, where the probability that an agent transitions state at time $t$
can depend on the state of its neighbors at time $t$. Thus, the dimension of the state space can be as large as $3^N$.

To reduce the size of the state space, we use an {\color{black} individual-based mean-field approach (IBMF)~\cite{pastor2015epidemic,van2009virus}}. This approach assumes that the state of every node is statistically independent from the state of its neighbors. The approximation reduces the state of every agent $j$ to the probabilities $p_j^S(t)$, $p_j^I(t)$, and $p_j^R(t)$ of agent $j$ being in state $S$, $I$, and $R$, respectively, at time $t \geq 0$. Since at every time $t\geq 0$, these probabilities sum to 1, the state of every agent $j$ evolves on the 2-simplex $\Delta := \{(p_j^S,p_j^I,p_j^R) \in [0,1]^3 | \ p_j^S+p_j^I+p_j^R = 1 \}$. The reduced state space corresponds to $N$ copies of $\Delta$, denoted $\Delta_N$, which has dimension $2N$. 

The dynamics retain the full topological structure of the network encoded in the infection and reinfection rates $\beta_{jk}$ and $\hat{\beta}_{jk}$, which depend on the entries of the adjacency matrix $A$. We refer the reader to~\cite{van2013inhomogeneous} for a detailed derivation of the individual mean-field approximation for the SIS model, and to~\cite{van2015accuracy,gleeson2012accuracy} for a discussion and numerical exploration of the accuracy of mean-field approximations in network dynamics. 

Under the individual mean-field approximation, the dynamics of the network SIRI model on {\color{black}$\Delta_N$} are given by
\begin{align} \label{eq:3N_Model}
\dot{p}_j^S =& - p_j^S \sum_{k=1}^N \beta_{jk} p_k^I \nonumber \\
\dot{p}_j^I =& -\delta_j p_j^I + p_j^S \sum_{k=1}^N \beta_{jk} p_k^I + p_j^R \sum_{k=1}^N \hat{\beta}_{jk}  p_k^I \nonumber \\
\dot{p}_j^R =& - p_j^R \sum_{k=1}^N \hat{\beta}_{jk} p_k^I + \delta_j p_k^I,
\end{align}
{\color{black}where $p_j^S(0) + p_j^I(0) + p_j^R(0) = 1$ for all $j$.} 

We can reduce the number of equations from $3N$ to $2N$ by using the substitution $p_j^R = 1 - p_j^S - p_j^I$, {\color{black} for all $j$}, in~\eqref{eq:3N_Model}:
\begin{align}\label{eq:pj}
\dot{p}_j^S =& - p_j^S \sum_{k=1}^N \beta_{jk} p_k^I  \\
\dot{p}_j^I =& \sum_{k=1}^N \big((1-p_j^S) \hat{\beta}_{jk} + p_j^S  \beta_{jk} \big)p_k^I -\delta_j p_j^I -p_j^I \sum_{j=1}^N \hat{\beta}_{jk} p_k^I, \nonumber 
\end{align}
{\color{black}where $p_j^S(0) + p_j^I(0) = 1 - p_j^R(0)$ for all $j$.}

The dynamics can be written in matrix form where $p^\Omega = [p_1^\Omega,\cdots,p_N^\Omega]^T$ and $P^\Omega= \textrm{diag}(p^\Omega)$ for $\Omega \in \{S,I\}$: 
\begin{align}\label{eq:Het-SIRI}
\dot{p}^S &=  - P^S B p^I \nonumber \\
\dot{p}^I &=  \big(B^*(p^S) - D \big) p^I - P^I \hat{B} p^I,
\end{align}
where
\begin{equation*}
    B^*(p^S) = (\mathbb{I}-P^{S})\hat{B} + P^{S} B
    \label{eq:B*}
\end{equation*}
and
\begin{align*}
    B &= \{\beta_{jk}\} \succ \mathbf{\bar{0}}  \;\;\;\; &(\textrm{infection matrix}), \\
    \hat{B}&=\{\hat{\beta}_{jk}\} \succeq \mathbf{\bar{0}}  \;\;\;\; &(\textrm{reinfection matrix}), \\
    D &= \textrm{diag}(\delta_1, \ldots, \delta_N) \succeq \mathbf{\bar{0}}  \;\;\;\; &(\textrm{recovery matrix}).
\end{align*}
Further, we define
\begin{equation}
 \bar{B}_{max} = [\max(\beta_{jk},\hat{\beta}_{jk})], \;\;\; \bar{B}_{min} = [\min(\beta_{jk},\hat{\beta}_{jk})].
 \label{eq:Bmaxmin}
\end{equation}

The network SIRI model dynamics provide sufficient richness to describe a family of models which can be classified into six different cases (summarized in Table~\ref{tab:SIRI_SpecialCases}):

\begin{itemize}
\item Case 1 (SI): When $D=\bar{\mathbf{0}}$ the network SIRI model specializes to the network SI model.
\item Case 2 (SIR): When $\hat{B}=\bar{\mathbf{0}}$, the network SIRI model specializes to the network SIR model.
\item Case 3 (SIS): When $B=\hat{B}$ the network SIRI model specializes to the network SIS model with $p^S+p^R \mapsto p^S$.
\item Case 4 (Partial Immunity): When ${B} \succ \hat{B} \succ \bar{\mathbf{0}}$, every recovered agent acquires partial (or no) immunity to each of its infected neighbors.
\item Case 5 (Compromised Immunity): When $\hat{B} \succ B \succ \bar{\mathbf{0}}$, every recovered agent acquires compromised (or no) immunity to each of its infected neighbors. 
\item Case 6 (Mixed Immunity): Models not in Cases 1-5. Notably, there is at least one pair of edges $(j,k)$ and $(l,m)$ such that $\beta_{jk} \geq \hat{\beta}_{jk}$ and $\beta_{lm} < \hat{\beta}_{lm}$. We classify mixed immunity into two sub-cases: 
\begin{itemize}
    \item Case 6a (Weak Mixed Immunity): For every agent $j$, $\beta_{jk} - \hat{\beta}_{jk} \geq 0$ for all $k \in \mathcal{N}_j$ or $\beta_{jk} - \hat{\beta}_{jk} \leq 0$ for all $k \in \mathcal{N}_j$. 
    \item Case 6b (Strong Mixed Immunity): Mixed immunity that is not weak.
\end{itemize}
\end{itemize}

\begin{table}[!t]
\caption{Network SIRI model cases}
\label{tab:SIRI_SpecialCases}
\begin{center}
\begin{tabular}{c|c|c}
\textbf{Case} & \textbf{Parameter Value} & \textbf{Equivalent Model}\\
\hline
1 & $D= \bar{\mathbf{0}}$ & SI\\
2 & $\hat{B}=\bar{\mathbf{0}}$ & SIR\\
3 & $B=\hat{B}$ & SIS\\
4 & $B \succ \hat{B} \succ \bar{\mathbf{0}}$ & Partial Immunity\\
5 & $\hat{B} \succ B \succ \bar{\mathbf{0}}$ & Compromised Immunity\\
6 & Otherwise & Mixed Immunity 
\end{tabular}
\end{center}
\end{table}

\section{Equilibria and Reproduction Numbers}\label{sec:EqlAndRepNum}

In this section we analyze the equilibria of the network SIRI model dynamics and define the notion of basic and extreme basic reproduction numbers. 
We denote the value of $p^S$ and $p^I$ at equilibrium as $p^{S*}$ and $p^{I*}$, respectively. 

\subsection{Equilibria}
\begin{prop}\label{prop:equilibria}
{\color{black} The only equilibria of the network SIRI model (3) are an invariant set of infection-free equilibria (IFE)  and one or more isolated endemic equilibria (EE). The IFE set is defined as $\mathcal{M} = \{(p^{S*},\mathbf{0}) \in \Delta_N | \ \mathbf{0} \preceq p^{S*} \preceq \mathbf{1} \}$, corresponding to all equilibria in which $p^{I*} = \mathbf{0}$, i.e., $p^{S*} + p^{R*}=\mathbf{1}$. The EE are defined as equilibria where $p^{I*} \succ \mathbf{0}$ satisfies
\begin{equation}\label{eq:EE_j}
p^{I*}_j = \frac{ \sum_{k=1}^N \hat{\beta}_{jk} p_k^{I*}}{\delta_j + \sum_{k=1}^N \hat{\beta}_{jk} p_k^{I*}}.
\end{equation}
  If $\hat{B}$ is irreducible then, for every EE, $p^{S*} = \mathbf{0}$ and $p^{I*} \gg \mathbf{0}$.}
\end{prop}
\begin{proof}
Setting $\dot{p}^S=\mathbf{0}$ in~\eqref{eq:Het-SIRI}, we get $P^{S*} B p^{I*}=\mathbf{0}$. Since $\mathcal{G}$ is strongly connected and $B$ preserves the connectivity of $A$, then for every agent $j$ we must have $p_j^{S*} = 0$ or $\sum_{\mathcal{N}_j} p_k^{I*}=0$. Moreover, since $\hat{B}$ has a zero at every entry where $B$ has a zero, it follows that $P^{S*} \hat{B} p^{I*}=\mathbf{0}$.

Setting $\dot{p}^I=\mathbf{0}$ in~\eqref{eq:Het-SIRI} and using $P^{S*} B p^{I*} = P^{S*} \hat{B} p^{I*}=\mathbf{0}$ we get
\begin{equation} \label{eq:EqPtpI}
\mathbf{0} =(\hat{B}-D - P^{I*}\hat{B}) p^{I*} =(\hat{B} - D - \textrm{diag}(\hat{B}p^{I*}))p^{I*}.
\end{equation}
One solution is the invariant set  $\mathcal{M} = \{(p^{S*},\mathbf{0}) \in \Delta_N | \ \mathbf{0} \preceq p^{S*} \preceq \mathbf{1} \}$. The only other solutions are isolated equilibria  $p^{I*} \succ \mathbf{0}$ satisfying \eqref{eq:EE_j}.

If $\hat{B}$ is irreducible, $\beta_{jk}>0$ for any $(j,k) \in \mathcal{E}$, and if $p_k^{I*} > 0$ for any $k \in \mathcal{V}$, then by~\eqref{eq:EE_j}  $p_j^{I*} > 0$ for any $j$ where $k \in \mathcal{N}_j$. So $p_i^{I*}>0$ for any $i$ where $j \in \mathcal{N}_i$. This argument can be recursively applied until all nodes in $\mathcal{G}$ are covered. Since $P^{S*} B p^{I*}=\mathbf{0}$, $p^{I*} \gg \mathbf{0}$ implies $p^{S*}=\mathbf{0}$.
\end{proof}

{\color{black}
\begin{prop}\label{def:MboundaryAndInt}
The boundary of $\mathcal{M}$ is $\partial\mathcal{M} = \{x = (p^{S*},\mathbf{0})\in \mathcal{M}\, | \, \exists j, \, p_j^{S*} \in \{0,1\} \}$. The corner set of $\mathcal{M}$ is $\hat{\mathcal{M}} = \{x = (p^{S*}, \mathbf{0}) \in \partial \mathcal{M} \, | \, p_j^{S*} \in \{0,1\}, \, \forall j \}$. The interior of $\mathcal{M}$ is $\textrm{int}(\mathcal{M}) = \mathcal{M} \setminus \partial \mathcal{M}$.
\end{prop}
\begin{proof}
The proof follows from the definition of $\mathcal{M}$.
\end{proof}
}
\begin{remark}\label{remark:SIRI_SIS_equivalence}
{\color{black} For $\hat B$ irreducible, the} equilibria of the network SIRI model are equivalent to the equilibria of the network SIS model (Case 3), where $B = \hat{B}$. This follows since any equilibrium of~\eqref{eq:Het-SIRI} satisfies~\eqref{eq:EqPtpI}, and therefore is also an equilibrium of the network SIS dynamics~\cite{lajmanovich1976,van2013inhomogeneous,fall2007epidemiological}: 
\begin{equation}\label{eq:SIS}
\dot{p}^I = (B - D) p^I -  P^I B p^I.
\end{equation}
For the network SI model (Case 1), the only equilibrium is a unique EE with $p^{I*} = \mathbf{1}$ and $p^{S*} = \mathbf{0}$. For the network SIR model (Case 2), the only equilibria are the IFE set $\mathcal{M}$.
\end{remark}

\begin{remark}
For initial conditions $p^S(0) = \mathbf{1}- p^I(0)$, the network SIRI dynamics~\eqref{eq:Het-SIRI} initially behave as the network SIS model~\eqref{eq:SIS} with infection matrix $B$. As agents become exposed to the infection for the first time, the dynamics transition to network SIS dynamics with infection matrix $\hat B$. 
\end{remark}

In the remainder of this paper, we assume $D $ is nonsingular and  $\hat B$ is irreducible; thus every EE is {\em strong} since $p^{I*} \gg \mathbf{0}$. The generalization to reducible $\hat{B}$ is straightforward.\footnote{The graph $\mathcal{G}_{\hat B}$ with reducible adjacency matrix $\hat{B}$  is weakly connected or disconnected. If $\mathcal{G}_{\hat B}$ is weakly connected, the adjacency matrix of $\mathcal{G}_{\hat B}$ can be written as an upper block triangular matrix with $K$ diagonal irreducible blocks that describe the $K$ strongly connected subgraphs of $\mathcal{G}$~\cite{khanafer2016}. If $\mathcal{G}_{\hat B}$ is disconnected, it is sufficient to study each connected subgraph of $\mathcal{G}_{\hat B}$.}

\subsection{Basic Reproduction Numbers}\label{sec:BasicRepNnum}
{\color{black}The basic reproduction number is a key concept in epidemiology, defined as the expected number of new cases of infection caused by a typical infected individual in a population of susceptible individuals~\cite{anderson1992infectious,diekmann1990definition,van2002reproduction}. For many deterministic epidemiological models, an infection can invade and persist in a fully susceptible population if and only if $R_0 > 1$~\cite{hethcote2000mathematics}. For the SIS model in a well-mixed population in which all agents have the same infection rate $\beta$ and recovery rate $\delta$, the basic reproduction number, $R_0 = \beta/\delta$, governs the steady-state behavior of solutions. If $R_0 \leq 1$, the infection eventually dies out. If $R_0>1$, the infection spreads through the population, converging to an endemic steady-state solution with a fixed fraction of the population in the infected state. In the network SIS model (7), the basic reproduction number is ${R_0 = \rho(B \Gamma^{-1})}$~\cite{van2002reproduction,fall2007epidemiological, kamgang2008computation}. In this case, if $R_0 \leq 1$, solutions reach the IFE set $\mathcal{M}$ as $t \to \infty$ while if $R_0 > 1$, solutions reach the unique EE~\eqref{eq:EE_j} as $t \to \infty$~\cite{lajmanovich1976,fall2007epidemiological}.}

In previous work~\cite{pagliaraSIRIWM} we proved that, in well-mixed settings, the transient and steady-state behavior of solutions in the SIRI model depend on two numbers $R_0$ and $R_1$, corresponding to the basic reproduction number for a population of susceptible individuals and for a population of recovered individuals, respectively. Here we extend the definition of $R_0$ and $R_1$ in~\cite{pagliaraSIRIWM} to network topologies and introduce the notion of extreme basic reproduction numbers.

{\color{black}
\begin{definition}[Basic and Extreme Basic Reproduction Numbers]\label{def:R}
Consider the following system, which is a modification to dynamics~(\ref{eq:Het-SIRI}):
\begin{align}\label{eq:Het-SIRI-mod}
\dot{p}^S &=  - P^S B p^I \nonumber \\
\dot{p}^I &=  \big(B^*(p) - D \big) p^I - P^I \hat{B} p^I,
\end{align}
where
\begin{equation*}
    B^*(p) = (\mathbb{I}-P)\hat{B} + P B,
\end{equation*}
$\mathbf{0} \preceq p \preceq \mathbf{1}$ is constant, and $P= \textrm{diag}(p)$.  Let $R(p)$ be the basic reproduction number for~(\ref{eq:Het-SIRI-mod}). We distinguish $R(p)$ for four different values of $p$ as follows:
\begin{itemize}
    \item {\em Basic infection reproduction number} $R_0$ is the reproduction number for (\ref{eq:Het-SIRI-mod}) with $p =  \mathbf{1}$.
    \item {\em Basic reinfection reproduction number} $R_1$ is the reproduction number for (\ref{eq:Het-SIRI-mod}) with $p = \mathbf{0}$.
    \item {\em Maximum basic reproduction number} $R_{max}$ is the reproduction number for (\ref{eq:Het-SIRI-mod}) with $p = \argmax_{\mathbf{0} \preceq q \preceq \mathbf{1}} R(q)$.
    \item {\em Minimum basic reproduction number} $R_{min}$ is the reproduction number for (\ref{eq:Het-SIRI-mod}) with $p = \argmin_{\mathbf{0} \preceq q \preceq \mathbf{1}} R(q)$.
\end{itemize}
\end{definition}

\begin{prop}[Spectral Radius Formulas for Reproduction Numbers]
The basic reproduction number $R(p)$ of~(\ref{eq:Het-SIRI-mod}) for $\mathbf{0} \preceq p \preceq \mathbf{1}$ can be computed as
\[
R(p) = \rho(B^*(p) D^{-1}).
\]
Therefore,
\[
R_0  =  \rho(BD^{-1}), \;
R_1 =  \rho(\hat{B}D^{-1}), 
\]
\begin{align*}
R_{max} =  \max_{\mathbf{0} \preceq p \preceq \mathbf{1}} \rho(B^*(p)D^{-1}), \;
R_{min} =  \min_{\mathbf{0} \preceq p \preceq \mathbf{1}} \rho(B^*(p)D^{-1}).
\end{align*}
\end{prop}
\begin{proof}
The linear term of the dynamics of $p^I$ in~\eqref{eq:Het-SIRI-mod} is $K = B^*(p) - D$. Thus, $K$ is Metzler and by Definition~\ref{def:regsplitting} $K = T + U$ is a regular splitting where $T = B^*(p)$ and $U = -D$. Following~\cite{kamgang2008computation}, $R(p)= \rho(-TU^{-1}) =  \rho(B^*(p) D^{-1})$. 
\end{proof}


}



\begin{prop}[Reproduction number ordering] \label{prop:RmaxRmin}
Let  $\bar{R}_{max} = \rho(\bar{B}_{max}D^{-1})$ and $\bar{R}_{min} = \rho(\bar{B}_{min}D^{-1})$. Then,
\begin{equation}\label{eq:Rbounds}
    \bar{R}_{min} \leq R_{min} \leq \lambda_{max}(B^*(p^{S})D^{-1}) \leq R_{max} \leq \bar{R}_{max}
\end{equation}
for any $\mathbf{0} \preceq p^{S} \preceq \mathbf{1}$. If $B \succeq \hat{B}$, then $ \bar{R}_{max} = R_{max} = R_0$ and $\bar{R}_{min} = R_{min} = R_1$. If $\hat{B} \succeq B$, then $\bar{R}_{max} = R_{max} = R_1$, and $\bar{R}_{min} = R_{min} = R_0$.
\end{prop}
\begin{proof}
Any matrix with nonnegative entries is Metzler. Thus, $Y(p^{S}) = B^*(p^{S})D^{-1}$ is an irreducible Metzler matrix since  $B \succ \bar{\mathbf{0}}$ and $\hat{B} \succeq \bar{\mathbf{0}}$ are irreducible. By Proposition~\ref{prop:Metzler}, $\lambda_{max}(Y(p^{S}))$ increases (decreases) as any entry in $Y(p^{S})$ increases (decreases). Since every non-zero entry of $Y(p^{S})$ is a scaled convex sum of $\beta_{jk}$ and $\hat{\beta}_{jk}$, it follows that
$\bar{B}_{min}D^{-1} \preceq Y(p^{S}) \preceq \bar{B}_{max}D^{-1}$ for any $\mathbf{0} \preceq p^{S} \preceq \mathbf{1}$.
Consequently, \eqref{eq:Rbounds} holds for any $\mathbf{0} \preceq p^S \preceq \mathbf{1}$. If $B \succeq \hat{B}$, then $\bar{B}_{max}=B$ and $\bar{B}_{min}=\hat{B}$. If $\hat{B} \succeq B$, then $\bar{B}_{max} =\hat{B}$ and $\bar{B}_{min} =B$. The stated results then follow from the definitions of the  reproduction numbers.
\end{proof}
\section{Stability of Equilibria}\label{sec:Stability}
In this section we prove conditions for the local stability of the EE and of points in the IFE set $\mathcal{M}$.

\subsection{Stability of the Endemic Equilibria}
\begin{prop}\label{prop:EEstrong}
The network SIRI dynamics~\eqref{eq:Het-SIRI} have a unique EE, {\color{black}given by~\eqref{eq:EE_j}, which exists} if and only if $R_1 > 1$. {\color{black}When it exists,} the EE is locally stable. 
\end{prop}
\begin{proof}
{\color{black} 
Since $B$ and $\hat{B}$ are irreducible, $p^{S*}=\mathbf{0}$.
Following the proof of Proposition~\ref{prop:equilibria} and the argument of Remark~\ref{remark:SIRI_SIS_equivalence}, the equilibria of the network SIRI model~\eqref{eq:Het-SIRI} are equivalent to the equilibria of the network SIS model~\eqref{eq:SIS} where $B = \hat{B}$.
Thus, the proof of existence and uniqueness of the EE of~\eqref{eq:Het-SIRI}, if and only if $R_1 > 1$, follows the proof in Section 2.2 of~\cite{fall2007epidemiological} for existence and uniqueness of the EE of~\eqref{eq:SIS} where $B = \hat{B}$.
}


 To prove local stability, we compute the Jacobian  of~\eqref{eq:Het-SIRI} at the {\color{black}unique} EE~\eqref{eq:EE_j}:
\begin{equation} \label{eq:Het-SIRI_LinEE} 
J_{EE} = 
\begin{bmatrix}
- \textrm{diag}(B p^{I*})
& \bar{\mathbf{0}} \\ 
\ \textrm{diag}((B-\hat{B}) p^{I*}) & J_a
\end{bmatrix},
\end{equation}
where $J_a = \hat{B} - D -P^{I*} \hat{B} - \textrm{diag}(\hat{B} p^{I*})$. Since 
$-\textrm{diag}(Bp^{I*})$ is Hurwitz, showing that the EE is locally stable is equivalent to showing that the Metzler matrix $J_a$ is Hurwitz. 

By~\eqref{eq:EqPtpI},  $(\hat{B}-D -P^{I*}\hat{B}) p^{I*} = 0$, and thus
\begin{equation}
J_a p^{I*} = - \textrm{diag}(\hat{B}p^{I*})p^{I*} \ll \mathbf{0},
\end{equation}
where the inequality follows from $p^{I*} \gg \mathbf{0}$.
By Proposition~\ref{prop:MetzlerHurwitz} we conclude that $J_a$ is Hurwitz.
\end{proof}

\subsection{Stability of Infection-Free Equilibria}\label{sec:StabIFE}
In this section we prove results on the stability of the IFE set $\mathcal{M}$. The equilibria in $\mathcal{M}$ are non-hyperbolic: the Jacobian of \eqref{eq:Het-SIRI} at a point $x\in\mathcal{M}$ has $N$ zero eigenvalues corresponding to the $N$-dimensional space tangent to $\mathcal{M}$. The remaining $N$ eigenvalues are called {\em transverse} as they correspond to the $N$-dimensional space transverse to $\mathcal{M}$. The Shoshitaishvili Reduction Principle~\cite{shoshitaishvili1975bifurcations}, which extends the Hartman-Grobman Theorem to non-hyperbolic equilibria, can be used to study the local stability of points in $\mathcal{M}$ in terms of the transverse eigenvalues of the Jacobian and the dynamics on the center manifold. We show how the irreducibility of $B$ and $\hat{B}$ imply that the behavior of solutions in $\Delta_N$ close to a point $x \in \mathcal{M}$ depends only on the sign of the leading transverse eigenvalue of the Jacobian at $x$. 

Throughout the rest of this paper, we consider the topological space $\Delta_N$ as a subspace of $\mathbb{R}^{2N}$. This allows us to study points in $\partial \mathcal{M}$ and in $\textrm{int}(\mathcal{M})$ simultaneously. In $\mathbb{R}^{2N}$, the invariant set $\mathcal{M}$ of IFE points becomes a subset of the invariant manifold of equilibria $\mathcal{M}' = \{ (p,\mathbf{0}) | p \in \mathbb{R}^N \}$. 


\begin{lemma}[Local Stability of Points in the IFE set $\mathcal{M}$]\label{lemma:M}
Let $x = (p^{S*},\mathbf{0}) \in  \mathcal{M}$. Let $J_{\mathcal{M}}(x)$ be the Jacobian of \eqref{eq:Het-SIRI} at $x$ and $\lambda_{Tmax}(J_{\mathcal{M}}(x))$  the leading transverse eigenvalue of $J_{\mathcal{M}}(x)$. 
Then, $\lambda_{Tmax}(J_{\mathcal{M}}(x)) \in \mathbb{R}$ and the following hold.  
\begin{itemize}
\item Suppose $\lambda_{Tmax}(J_{\mathcal{M}}(x)) < 0$. Then, $x$ is locally stable.
I.e., given a neighborhood $U$ of $x$ on $\mathcal{M}'$ such that $\lambda_{Tmax}(J_{\mathcal{M}}(u))<0$ for all $u \in U$, there exists $V \subset \Delta_N$ and $x \in V$ 
such that any solution starting in $V$ converges exponentially to a point in $U \cap \Delta_N$. 
\item Suppose $\lambda_{Tmax}(J_{\mathcal{M}}(x))>0$. Then, $x$ is unstable. I.e., there exists  $W \subset \Delta_N$ and $x\in W$, such that any solution starting in $W$ leaves $W$.
\end{itemize}
\end{lemma}

\begin{proof}
For an arbitrary point $x = (p^{S*},\mathbf{0}) \in \mathcal{M}$, 
\begin{equation}\label{eq:J_IFE}
J_{\mathcal M}(x) = 
\begin{bmatrix}
\bar{\mathbf{0}}  & - P^{S*} B \\ 
\bar{\mathbf{0}}  & J_T(p^{S*})
\end{bmatrix},
\end{equation}
where $J_T(p^{S*}) = B^*(p^{S*}) - D$. The $N$ transverse eigenvalues of $J_{\mathcal{M}}(x)$ are the eigenvalues of $J_T(p^{S*})$ and so  $\lambda_{Tmax}(J_{\mathcal{M}}(x))= \lambda_{max}(J_T(p^{S*}))$. The matrix $J_T(p^{S*})$ is Metzler irreducible since $B$ and $\hat{B}$ are Metzler irreducible. By Proposition~\ref{prop:Metzler} $\lambda_{max}(J_T(p^{S*})) \in \mathbb{R}$. 

Consider an arbitrary point $x' = (p',\mathbf{0}) \in \mathcal{M}'\setminus \mathcal{M}$. The Jacobian of~\eqref{eq:Het-SIRI} at $x'$ takes on the same form as~\eqref{eq:J_IFE}, and $J_T(p')$ is Metzler irreducible if every entry of $p'$ satisfies
\begin{equation*}
\begin{cases}
    p'_j > -\frac{\sum_{j=1}^N \hat{\beta}_{jk}}{\sum_{j=1}^N (\beta_{jk} - \hat{\beta}_{jk})} & \rm{if} \; \sum_{j=1}^N (\beta_{jk} - \hat{\beta}_{jk}) \geq 0,\\
    p'_j < \frac{\sum_{j=1}^N \hat{\beta}_{jk}}{\sum_{j=1}^N(\hat{\beta}_{jk} - \beta_{jk})} & \rm{if} \; \sum_{j=1}^N (\beta_{jk} - \hat{\beta}_{jk}) \leq 0.        
\end{cases}
\end{equation*}
Since $B,\hat{B}$ are irreducible, $\mid \sum_{j=1}^N \hat{\beta}_{jk}/\sum_{j=1}^N (\beta_{jk} - \hat{\beta}_{jk}) \mid > 1$
for all $j$. So, for any $x \in \partial \mathcal{M}$, there exists a neighborhood $\bar{U}$ of $x$ on $\mathcal{M}'$ such that $J_T(\bar{u}')$ is Metzler irreducible for every $\bar{u} = (\bar{u}',\mathbf{0}) \in \bar{U}$. By Proposition~\ref{prop:Metzler} $\lambda_{max}(J_T(p')) \in \mathbb{R}$.

Let $U$ be a neighborhood of $x$ on $\mathcal{M}'$ such that $\lambda_{Tmax}(J_{\mathcal{M}}(u))$ has the same sign as $\lambda_{Tmax}(J_{\mathcal{M}}(x))$ for all $u  \in U$. Then, $\lambda_{max}(J_{T}(u'))$ has the same sign as $\lambda_{max}(J_{T}(p^{S*}))$ for all $u = (u', \mathbf{0})\in U$. By Proposition~\ref{prop:Metzler}  every left and right eigenvector of every eigenvalue of $J_T(u')$, other than $\lambda_{max}(J_T(u'))$, contains at least one negative entry. Thus, for any $\bar{u} \in U \cap \Delta_N$, the eigenvector corresponding to $\lambda_{Tmax}(J_\mathcal{M}(\bar{u}))$ lies in $\Delta_N$, and the eigenvectors corresponding to the other $N-1$ transverse eigenvalues lie outside $\Delta_N$. 

If $\lambda_{max}(J_T(p^{S*}))<0$, then every transverse eigenvalue of $J_{\mathcal M}(x)$ has negative real part. By the Shoshitaishvili Reduction Principle~\cite{shoshitaishvili1975bifurcations}, there exists a neighborhood $V' \in \mathbb{R}^{2N}$ of $x$ that is positively invariantly foliated by a family of stable manifolds corresponding to the family of stationary solutions in $U$ (see~\cite{henry1981geometric,aulbach1984continuous,liebscher2015bifurcation}), each stable manifold spanned by the (generalized) eigenvectors associated with the $N$ negative transverse eigenvalues of $J_\mathcal{M}(u)$. Let $V = V' \cap \Delta_N$. Then $V \subset \Delta_N$ is positively invariantly foliated by a family of stable manifolds. The invariance of $\Delta_N$ implies each of these stable manifolds corresponds to a point $\bar{u} \in U \cap \mathcal{M}$. Thus, any solution starting in $V$ converges exponentially along a stable manifold 
to the corresponding stationary solution in $U \cap \mathcal{M}$.

If $\lambda_{max}(J_T(p^{S*}))>0$, then there is at least one transverse eigenvalue of $J_{\mathcal M}(x)$ with positive real part. The trace of $J_T(p^{S*})$ is negative for any ${\mathbf{0} \preceq p^{S*} \preceq \mathbf{1}}$, so the sum of the eigenvalues of $J_T(p^{S*})$ is always negative and $J_{\mathcal M}(x)$  has at least one transverse eigenvalue with negative real part. 
By the Shoshitaishvili Reduction Principle~\cite{shoshitaishvili1975bifurcations} there exists a neighborhood $W' \subset \mathbb{R}^{2N}$ of $x$ that is
positively invariantly foliated by a family of stable, unstable, and possibly center manifolds corresponding to the family of stationary solutions in $U$. Let $W = W' \cap \Delta_N$. Then, the stable and center manifolds of each stationary solution $\bar{u} \in W \cap \mathcal{M}$ lie outside $\Delta_N$. Thus, no solution starting in $W$ can remain in $W$ for all time, i.e., any solution starting in $W$ leaves $W$. 
\end{proof}

\begin{definition}[Stable, unstable, and center IFE subsets]\label{def:Msubsets}
The {\em stable IFE subset} is $\mathcal{M}_- = \{x \in \mathcal{M}\, |\, \lambda_{Tmax}(J_\mathcal{M}(x)) < 0 \}$. The {\em unstable IFE subset} is $\mathcal{M}_+ = \{x \in \mathcal{M}\, |\, \lambda_{Tmax}(J_\mathcal{M}(x)) > 0 \}$. The {\em center IFE subset} is $\mathcal{M}_0 = \{x \in \mathcal{M}\, |\, \lambda_{Tmax}(J_\mathcal{M}(x)) = 0 \}$. 
\end{definition}

\begin{prop}
${\mathcal{M}_- \cup \mathcal{M}_+ \cup \mathcal{M}_0 = \mathcal{M}}$. Every point in $\mathcal{M}_-$ is locally stable and every point in $\mathcal{M}_+$ is unstable.
\end{prop}
\begin{proof} This follows from Definition~\ref{def:Msubsets} and Lemma~\ref{lemma:M}.
\end{proof}

We now state {\color{black} the first two} theorems of the paper, which  relate the extreme basic reproduction numbers $R_{max}$ and $R_{min}$ to the stable, unstable, and center subsets of the IFE set $\mathcal{M}$.

\begin{theorem}[Stability of the IFE set $\mathcal{M}$]\label{thm:stability}\leavevmode
\begin{enumerate}[(A)]
    \item If $R_{max} < 1$, then $\mathcal{M}_- = \mathcal{M}$.\label{itm:IFE_stable}
    \item If $R_{min} > 1$, then $\mathcal{M}_+ = \mathcal{M}$.\label{itm:IFE_unstable}
    \item If $R_{max} = R_{min} = 1$, then $\mathcal{M}_0 = \mathcal{M}$.\label{itm:IFE_Rmax=Rmin=1} 
    \item If $R_{min} <R_{max} = 1$,  
    then $\mathcal{M}_- = \mathcal{M} \setminus \mathcal{M}_0$ and $\mathcal{M}_0 \subset \partial\mathcal{M}$.\label{itm:IFE_Rmax=1} 
    \item If $R_{max} > R_{min} = 1$, then $\mathcal{M}_+ = \mathcal{M} \setminus \mathcal{M}_0$
    and $\mathcal{M}_0 \subset \partial\mathcal{M}$.\label{itm:IFE_Rmin=1}
    \item If {\color{black} $R_{max} > 1 > R_{min}$}, then $\mathcal{M}_-,\mathcal{M}_+,\mathcal{M}_0 \neq \emptyset$ and each subset consists of $n_-,n_+,n_0$ connected sets, respectively. Each of the center connected sets $\mathcal{M}_0^j$, $j = 1,\dots,n_0$, is an $N-1$-dimensional smooth hypersurface with boundary $\partial \mathcal{M}_0^j \subset \partial \mathcal{M}$. Each $\mathcal{M}_0^j$ separates an $N$-dimensional stable connected hypervolume from an $N$-dimensional unstable connected hypervolume.\label{itm:IFE_mixed}
\end{enumerate}
\end{theorem}
 
\begin{remark}
 Theorem~\ref{thm:stability} applies to the six different cases of the network SIRI model as follows {\color{black}(see Table~\ref{tab:Stability_IFE})}. \ref{itm:IFE_stable} applies to Cases 2, 3, 4, 5, and 6. \ref{itm:IFE_unstable} applies to Cases 3, 4, 5, and 6. \ref{itm:IFE_Rmax=Rmin=1} applies to Case 3. \ref{itm:IFE_Rmax=1} applies to Cases 2, 4, 5, and 6. \ref{itm:IFE_Rmin=1} applies to Cases 4, 5, and 6. \ref{itm:IFE_mixed} applies to Cases 2, 4, 5, and 6. We specialize~\ref{itm:IFE_mixed} in Theorem~\ref{thm:uniqueMsubsets} to provide the key to characterizing global behavior in Cases 2, 4, 5, and 6a. 
\end{remark}

\begin{table}[!t]
\caption{Applicability of Theorem~\ref{thm:stability} statements\hspace{\textwidth} to the Network SIRI model cases}
\label{tab:Stability_IFE}
\begin{center}
\begin{tabular}{c|c|c|c|c|c|c|c}
\textbf{Case} & \textbf{Model} &\textbf{A} & \textbf{B} & \textbf{C} & \textbf{D} &\textbf{E} & \textbf{F}\\
\hline
1 & SI & &  &  &  & & \\
2 & SIR & \checkmark & & & \checkmark & & \checkmark \\
3 & SIS & \checkmark & \checkmark & \checkmark & & & \\
4 & Partial & \checkmark & \checkmark & & \checkmark  & \checkmark & \checkmark\\
5 & Comprom. & \checkmark & \checkmark & & \checkmark  & \checkmark & \checkmark\\
6 & Mixed & \checkmark & \checkmark & & \checkmark  & \checkmark & \checkmark\\
\end{tabular}
\end{center}
\end{table}
\begin{theorem}[Uniqueness of stable, unstable, and center subsets]\label{thm:uniqueMsubsets}
If {\color{black}$R_{max} > 1 > R_{min}$}, then for Case 2 (SIR), Case 4 (partial immunity), Case 5 (compromised immunity), and Case 6a (weak mixed immunity), $\mathcal{M}_0$ consists of a unique $(N-1)$-dimensional hypersurface with boundary $\partial\mathcal{M}_0 \subset \partial \mathcal{M}$ dividing $\mathcal{M}$ into $\mathcal{M}_-$ and $\mathcal{M}_+$.
\end{theorem}

\begin{remark}
 We conjecture that Theorem~\ref{thm:uniqueMsubsets} can be extended to Case 6b. Extensive computations of $\mathcal{M}_0$, $\mathcal{M}_-$, and $\mathcal{M}_+$, for an $N=3$ agent network, with different network configurations and parameter values, consistently show a unique connected surface $\mathcal{M}_0$ dividing $\mathcal{M}$ into $\mathcal{M}_-$ and $\mathcal{M}_+$.
\end{remark}

The proofs of Theorems~\ref{thm:stability} and~\ref{thm:uniqueMsubsets} make use of the following definition and lemmas.

{\color{black}\begin{lemma}[Neighborhood $E \subset \mathcal{M}'$ of $\mathcal{M}$]
Let $E$ be the union of $\mathcal{M}$ and the neighborhoods $\bar{U} \subset \mathcal{M}'$ of every $\bar{x} \in \partial \mathcal{M}$ described in the proof of Lemma~\ref{lemma:M}. Then, $E\subset \mathcal{M}'$ is a neighborhood of $\mathcal{M}$.
\end{lemma}
}


{\color{black}\begin{lemma}[$J_T$ and $\Lambda$]\label{lem:LambdaFnc}
Let $J_T(p) = B^*(p) - D$ for $(p,\mathbf{0}) \in E$ and
 $\Lambda: E \to \mathbb{R}$, $(p,\mathbf{0}) \mapsto
\lambda_{max}(J_T(p))$. For ease of notation we  use $\Lambda(p)$ for $\Lambda(p, \mathbf{0})$. Then $\Lambda(p)=\lambda_{Tmax}(J_{\cal M}(p, \mathbf{0}))$. 
\end{lemma}
}

\begin{definition}[Stubborn agents]
An agent $j \in \mathcal{V}$ is {\em stubborn} if $(\beta_{jk} -\hat{\beta}_{jk}) = 0$ for all $k \in \mathcal{N}_j$. 
\end{definition}

\begin{lemma}[IFE subsets as level surfaces of $\Lambda$]\label{lemma:MSubsetsAndLambda}
Let 
{\color{black}$\Lambda_c = \{(p,\mathbf{0}) \in E \, | \, \Lambda(p) = c\}$
be the level surface of $\Lambda$ on $E\subset \mathcal{M}'$ corresponding to $c \in \mathbb{R}$. Then,
$\mathcal{M}_0 = \Lambda_0 \cap \mathcal{M}$, $\mathcal{M}_- = \bigcup_{c < 0} \Lambda_c \cap \mathcal{M}$ and $\mathcal{M}_+ = \bigcup_{c > 0} \Lambda_c \cap \mathcal{M}$.
}
\end{lemma}
\begin{proof}
This follows from Definition~\ref{def:Msubsets} and Lemma~\ref{lem:LambdaFnc}. 
\end{proof}
\begin{lemma}[Gradient of $\Lambda$]\label{lemma:LambdaDyn}
For $\tilde{x} = (p,\mathbf{0})\in E$, let $w^T,v \in \mathbb{R}^N$ be left and right eigenvectors of $J_T(p)$ for $\Lambda(p)$. Then, $\Lambda$ is smooth on $E$, i.e., $\Lambda(\cdot) \in C^{\infty}(E)$, with partial derivatives
\begin{equation}\label{eq:LambdaDyn}
    \frac{\partial \Lambda}{\partial p_j}(p)  = w_j 
    \sum_{k=1}^N (\beta_{jk} - \hat{\beta}_{jk})v_k
\end{equation}
and gradient
\begin{equation}\label{eq:gradientLambda}
    \nabla \Lambda (p) =
    \textnormal{diag}(w)(B-\hat{B})v.
\end{equation}
In addition, the following hold.
\begin{itemize}
    \item If $B = \hat{B}$, all points $\tilde{x} \in E$ are critical points of $\Lambda$. 
    \item If $j$ is a stubborn agent, $(\partial \Lambda/\partial p_j)(p) = 0$ for all $\tilde{x} \in E$.
    \item If $B \succ \hat{B}$, and there are no stubborn agents in $\mathcal{G}$, $\nabla \Lambda (p) \gg \mathbf{0}$ for all $\tilde{x} \in E$, and $\Lambda$ has no critical points.
    \item If $\hat{B} \succ B$, and there are no stubborn agents in $\mathcal{G}$,  $\nabla \Lambda (p) \ll \mathbf{0}$ for all $\tilde{x} \in E$ and $\Lambda$ has no critical points. 
    \item If $B \neq \hat{B}$ then either $\Lambda$ has no critical points or all points $\tilde{x} \in E$ are critical points of $\Lambda$.
\end{itemize}
\end{lemma}
\begin{proof}
Let $\tilde{x} = (p,\mathbf{0}) \in E$. By the proof of Lemma~\ref{lemma:M}, $J_T(p)$ is Metzler irreducible. By Proposition~\ref{prop:Metzler}, $\lambda_{max}(J_T(p))\in \mathbb{R}$ and has multiplicity one. Thus by~\cite[p.~66-67]{wilkinson1965algebraic}, $\Lambda(\cdot) \in C^\infty(E)$.




Differentiating the right-eigenvector equation $J_T(p) v= \Lambda(p) v$ with respect to $p_j$, and premultiplying by  $w^T$ we get
\begin{equation*}
        w^T \frac{\partial J_T}{\partial p_j}v + w^T J_T  \frac{\partial v}{\partial p_j} = w^T v \frac{\partial \Lambda}{\partial p_j} +  w^T \Lambda \frac{\partial v}{\partial p_j},
\end{equation*}
where all terms are evaluated at $p$, $\partial J_T/\partial p_j (p) = \{ z_{kl} \}$ and
\begin{equation*}
z_{kl} = 
    \begin{cases}
    \beta_{kl} - \hat{\beta}_{kl} & k=j \ \textrm{and} \ (k,l) \in \mathcal{E}, \\
    0 & \textrm{otherwise}.
    \end{cases}  
\end{equation*}
Because  $\Lambda(p)$ has multiplicity one, we can always pick $w$ such that $w^Tv=1$. Using $w^T J_T(p) = w^T\Lambda(p)$, we obtain
\begin{equation}\label{eq:PartialsLambda}
          \frac{\partial \Lambda}{\partial p_j}(p)  = 
          w^T \frac{\partial J_T}{\partial p_j}(p) v = w_j
          \sum_{k=1}^N (\beta_{jk} - \hat{\beta}_{jk})v_k.
\end{equation}
In vector form, \eqref{eq:PartialsLambda} becomes $\nabla \Lambda (p) =
    \textnormal{diag}(w)(B-\hat{B})v.$

By Proposition~\ref{prop:Metzler} $w,v\gg\mathbf{0}$. If $B =\hat{B}$, then $\nabla \Lambda (p) =\mathbf{0}$ for all $\tilde{x} \in E$. If $j$ is a stubborn agent, then $\beta_{jk}-\hat{\beta}_{jk} =0$ for all $k \in \mathcal{N}_j$. Therefore, $(\partial \Lambda/\partial p_j)(p) = 0$ for all $\tilde{x} \in E$. If $B \succ \hat{B}$ and there are no stubborn agents in $\mathcal{G}$, then $\sum_{k=1}^N (\beta_{jk} - \hat{\beta}_{jk}) v_k > 0$ for any $j$. Therefore, $\nabla \Lambda (p) \gg \mathbf{0}$ for all $\tilde{x} \in E$.
A similar argument is made if $\hat{B} \succ B$, and there are no stubborn agents in $\mathcal{G}$. 

By Proposition~\ref{prop:gradSys} and~\eqref{eq:gradientLambda}, $\tilde{x}$ is a critical point of $\Lambda$ 
if and only if $(B-\hat{B})v = \mathbf{0}$. 
Let $\tilde{x}_0 = (p_0,\mathbf{0}) \in E$ be a critical point, then $(B-\hat{B})v_0 = \mathbf{0}$. 
Thus, for any $p$, we have
\begin{align}\label{eq:lemma3proof}
    J_T(p)v_0 &= (B^*(p)-D)v_0 \nonumber \\
    &= (\hat{B} - D)v_0 + \textrm{diag}(p) (B -\hat{B})v_0 \nonumber \\
    &= (\hat{B} - D)v_0.
\end{align}
In particular, $(\hat{B} - D)v_0 = \Lambda(p_0)v_0$. Hence, $(\Lambda(p_0),v_0)$ is an eigenpair of $\hat{B}-D$ and by~\eqref{eq:lemma3proof} also an eigenpair of $J_T(p)$ for any $p$. 
Since $\Lambda(p)$ has multiplicity one and $v$ is the only eigenvector satisfying $v \gg \mathbf{0}$, then $(\Lambda(p),v)=(\Lambda(p_0),v_0)$ and   $\tilde{x} \in E$ is a critical point of $\Lambda$.
\end{proof}

\begin{lemma}[$\Lambda_c$ with stubborn agents]\label{lemma:stubbornlevelsurfaces}
If agents $1,\dots,k$, $k<N$, are stubborn, then at any $(p,\mathbf{0}) \in E$, the subspace spanned by $\{e^1,\dots,e^k\}$ is tangent to the level surface $\Lambda_c$ with $p \in \Lambda_c$. 
\end{lemma}
\begin{proof}
By Lemma~\ref{lemma:LambdaDyn}, for any stubborn agent $j$, $(e^j)^T \nabla \Lambda(p) = 0$ for any $(p,\mathbf{0}) \in E$. By Proposition~\ref{prop:gradSys}, $e^j$ is tangent to the level surface $\Lambda_c$ with $p \in \Lambda_c$. 
\end{proof} 

\begin{lemma}[Maximum and minimum values of $\Lambda$ on $\mathcal{M}$]\label{lemma:LambdaMaxMin}
$\Lambda$ achieves its global maximum and minimum $c_{max},c_{min}$ on $\mathcal{M}$ at one or more points in $\partial \mathcal{M}$. In Cases 2, 4, 5, and 6a, there exist unique corner points $(p^{max},\mathbf{0}) \in \hat{\mathcal{M}}$,  $(p^{min},\mathbf{0}) \in \hat{\mathcal{M}}$ such that $\Lambda(p^{max}) = c_{max}$, $\Lambda(p^{min}) = c_{min}$, $R_{max} = \rho(B^*(p^{max})D^{-1})$, and $R_{min} = \rho(B^*(p^{min})D^{-1})$. Moreover, $p^{max}$ and $p^{min}$ are the respective unique global maximum and minimum points of $\Lambda$ on $\mathcal{M}$ if and only if there are no stubborn agents in $\mathcal{G}$.
\end{lemma}
\begin{proof}
The case in which every point in $E$ is a critical point of $\Lambda$ is trivial.  Assume $\Lambda$ has no critical points in $E$. Since $\mathcal{M} \subset \mathcal{M}'$ is a compact set and $\Lambda$ is continuous, then by the Extreme Value Theorem, $\Lambda$ achieves 
$c_{max}$ and $c_{min}$ at one or more points in $\partial \mathcal{M}$. Let $(p^{max},\mathbf{0}) \in \partial \mathcal{M}$ and $(p^{min},\mathbf{0}) \in \partial \mathcal{M}$ be points such that $\Lambda(p^{max})=c_{max}$ and $\Lambda(p^{min})=c_{min}$. 

Let $(p^S,\mathbf{0}) \in \mathcal{M}$. Assume there are no stubborn agents. The components $(J_T(p^S))_{jk} = (1-p_j^S) \hat{\beta}_{jk} + p_j^S \beta_{jk}$, $j \neq k$, are maximized and minimized at $p_j^S =0$ or $p_j^S = 1$ for all $j$. If $B \succ \hat{B}$ (Cases 2 and 4), the entries $(J_T(p^S))_{jk}$ are maximized at $p_j^S=1$ and minimized at $p_j^S=0$ for all $j$. It follows from Proposition~\ref{prop:Metzler} that $p^{max} = \mathbf{1} \in \hat{\mathcal{M}}$ and $p^{min} = \mathbf{0} \in \hat{\mathcal{M}}$. Using a similar argument if $\hat{B} \succ B$ (Case 5) $p^{max} = \mathbf{0} \in \hat{\mathcal{M}}$ and $p^{min} = \mathbf{1} \in \hat{\mathcal{M}}$.  Further, in Case 6a with no stubborn agents, $\beta_{jk} - \hat{\beta}_{jk} > 0$ for all $k \in \mathcal{N}_j$ or $\beta_{jk} - \hat{\beta}_{jk} < 0$ for all $k \in \mathcal{N}_j$. Thus, similarly, $p_j^{max},p_j^{min} \in \{0,1\}$ for all $j$.
For any $p^S \neq p^{max},p^{min}$, it follows that $B^*(p^{max})D^{-1} \succ B^*(p^S)D^{-1} \succ B^*(p^{min})D^{-1}$. By Proposition~\ref{prop:Metzler}, $R_{max} = \rho(B^*(p^{max})D^{-1})$ and $R_{min} = \rho(B^*(p^{min})D^{-1})$.

Now suppose there are stubborn agents $1,\dots,k$, $k<N$. Let $v_e \in {\rm span}\{e^1,\dots,e^k \}$. Then, by Lemma~\ref{lemma:stubbornlevelsurfaces}, if $p^{max} \in \Lambda_{c_{max}}$ then $p^{max} + v_e \in \Lambda_{c_{max}}$ as long as $(p^{max} + v_e,\mathbf{0}) \in E$. Similarly, if $p^{min} \in \Lambda_{c_{min}}$ then $p^{min} + v_e \in \Lambda_{c_{min}}$ as long as $(p^{min} + v_e,\mathbf{0}) \in E$. 
\end{proof}

We now present the proof for Theorems~\ref{thm:stability} and~\ref{thm:uniqueMsubsets}.

\begin{proof}[Proof of Theorem~\ref{thm:stability}]
To prove~\ref{itm:IFE_stable}, assume $R_{max} < 1$. By Proposition~\ref{prop:MetzlerRS}, $\lambda_{Tmax}(J_\mathcal{M}(x)) < 0$ for all $x \in \mathcal{M}$. Therefore, $\mathcal{M}_- = \mathcal{M}$, $\mathcal{M}_+ = \emptyset$, and $\mathcal{M}_0 = \emptyset$. The proof of~\ref{itm:IFE_unstable} and~\ref{itm:IFE_Rmax=Rmin=1} follow similarly. 
To prove~\ref{itm:IFE_Rmax=1}, assume $R_{min} < R_{max} = 1$. By Proposition~\ref{prop:MetzlerRS}, $\max_{x \in \mathcal{M}} \lambda_{Tmax}(J_\mathcal{M}(x)) = 0$. Therefore, $\mathcal{M}_0 \neq \emptyset$ and $\mathcal{M}_- = \mathcal{M} \setminus \mathcal{M}_0$. By Lemma~\ref{lemma:LambdaMaxMin}, $\mathcal{M}_0 \subset \partial \mathcal{M}$.
The proof of~\ref{itm:IFE_Rmin=1} follows similarly.

To prove~\ref{itm:IFE_mixed}, assume {\color{black}$R_{max} > 1 > R_{min}$}. By Proposition~\ref{prop:MetzlerRS} and Lemma~\ref{lemma:LambdaMaxMin}, there exist points $x_{max} = (p_{max}^S,\mathbf{0})$ and $x_{min} = (p_{min}^S,\mathbf{0})$ in $\partial \mathcal{M}$ such that $\lambda_{Tmax}(J_\mathcal{M}(x_{max})) >0$ and $\lambda_{Tmax}(J_\mathcal{M}(x_{min})) < 0$. By the continuity of $\Lambda$ on $E$ 
and Lemma~\ref{lemma:MSubsetsAndLambda}, it follows that $\mathcal{M}_-,\mathcal{M}_+,\mathcal{M}_0 \neq \emptyset$, with each subset consisting of $n_-,n_+$, and $n_0$ connected sets, respectively, where each of the $n_0$ center sets 
separates $n_-$ stable connected sets from $n_+$ unstable connected sets. 


Since, by Lemma~\ref{lemma:LambdaDyn}, $\Lambda$ has no critical points in $E$, it follows that $c$ is a regular value of $\Lambda$ for any $c \in \mathbb{R}$. Hence, by the Implicit Function Theorem, every center connected set $M_0$ is an $(N-1)$-dimensional smooth hypersurface, and every stable and unstable connected set is an $N$-dimensional hypervolume.

Since $\Lambda$ has no critial points in $E$, the gradient dynamics of $\Lambda$ have no equilibria in $\mathcal{M} \subset E$. Thus, no center connected set $M_0$ in $\mathcal{M}_0$ is compact in $\textrm{int} (\mathcal{M})$ in any direction, since by Proposition~\ref{prop:gradSys} there cannot be any $\alpha$-limit points or $\omega$-limit points in $\mathcal{M}$. Thus, $\partial M_0$ in $\mathcal{M}$ must be contained in $\partial \mathcal{M}$.
\end{proof} 

\begin{proof}[Proof of Theorem~\ref{thm:uniqueMsubsets}]
Let {\color{black}$R_{max} > 1 > R_{min}$}. In Cases 2, 4, 5 and 6a, by Lemma~\ref{lemma:LambdaMaxMin}, there exists a unique corner point $x^{min} = (p^{min},\mathbf{0}) \in \hat{\mathcal{M}}$ where $\Lambda(p^{min}) \leq \Lambda(p^S)$ for any $x = (p^S,\mathbf{0}) \in \mathcal{M}$, and $R_{min} = \rho(B^*(p^{min})D^{-1})$. {\color{black}By Definition~\ref{def:R}} and Proposition~\ref{prop:MetzlerRS}, it follows that $\Lambda(p^{min}) < 0$.

Given $p^S$, there exists a point $(p^{end},\mathbf{0}) \in \partial \mathcal{M}$ such that the line connecting $p^{min}$ to $p^{end}$ passes through $p^S$.  Each point on the line is parameterized by $r \in [0,1]$ as follows: $s(r,p^{end}) = \{(p(r),\mathbf{0}) \in \mathcal{M} \, | \, p(r) = (1-r) p^{min} + r p^{end}\}$.

Let $\bar{p} = p^{end} - p^{min}$.  Then, $\bar{p}$ is tangent to $s(r,p^{end})$, and $\bar{p}_j \geq 0$ if $p^{min}_j = 0$ and $\bar{p}_j \leq 0$ if $p^{min}_j = 1$, for all $j$. By Proposition~\ref{prop:gradSys}, the directional derivative of $\Lambda$ at $s(r,p^{end})$ in the direction $\bar{p}$ is $D_{\bar{p}}\Lambda(s(r,p^{end})) = \bar{p}^T \nabla \Lambda(s(r,p^{end}))$.

By Lemma~\ref{lemma:LambdaMaxMin}, if $B \succ \hat{B}$ (Cases 2 and 4), then $p^{min} = \mathbf{0}$ and $\bar{p} \succ \mathbf{0}$. By Lemma~\ref{lemma:LambdaDyn}, $\nabla \Lambda(s(r,p^{end})) \succ 0$ and so $\bar{p}^T \nabla\Lambda(s(r,p^{end})) \geq 0$ for any $r \in [0,1]$. Similarly, for Cases 5 and 6a, $\bar{p}^T \nabla\Lambda(s(r,p^{end})) \geq 0$ for all $r\in [0,1]$.

If $\bar{p}^T \nabla\Lambda(s(r,p^{end})) = 0$, $\Lambda$ is constant along the line $s(r,p^{end})$, $r\in[0,1]$, and the line describes a level surface $\Lambda_c$. Moreover, since $\Lambda(p^{min}) < 0$, by Lemma~\ref{lemma:MSubsetsAndLambda} $c \neq 0$, i.e., $\mathcal{M}_0$ does not intersect $\Lambda_c$. For all other lines, $\bar{p}^T \nabla\Lambda(s(r,p^{end})) > 0$, which implies that $\Lambda$ is strictly increasing from a negative value at the corner $(p^{min},\mathbf{0})$ to the value at $(p^{end},\mathbf{0})\in \partial\mathcal{M}$.  By Theorem~\ref{thm:stability}, $\partial \mathcal{M}_0 \subset \partial \mathcal{M}$. Thus, there is only the possibility of a single crossing of $\Lambda_0$ on each of these lines. By Lemma~\ref{lemma:MSubsetsAndLambda} there is a unique center connected hypersurface.  
\end{proof}

\section{Reproduction Numbers Predict Behavioral Regimes}\label{sec:DynamicRegimes}

In this section we prove our third theorem, which provides conditions
that determine whether solutions of \eqref{eq:Het-SIRI} converge to a point in the IFE set $\mathcal{M}$ or to the EE as $t \to \infty$. We show that the basic and extreme basic reproduction numbers $R_0,R_1,R_{min},R_{max}$ distinguish four behavioral regimes in the network SIRI model, each characterized by qualitatively different transient and steady-state behaviors.

\subsection{The $\omega$-limit Set of Solutions}
The components of $p^S$ decrease monotonically along solutions of~\eqref{eq:Het-SIRI}. Here we show that this monotonicity implies that all solutions either converge to a point in $\mathcal{M}$ or to the EE as $t \to \infty$. Moreover, this means that when the EE is not an equilibrium of the dynamics, the infection cannot survive in the network and all solutions reach an IFE point in $\mathcal{M}$. The results in this section are valid even if $\hat{B}$ is not irreducible.

\begin{lemma}\label{lemma:IFEorEE_as_t2infty}
  Let $y(t,y_0) = (p^S(t),p^I(t))$ be the solution of~\eqref{eq:Het-SIRI} with initial condition  $y_0 \in \Delta_N$. Then the following hold:
  \begin{itemize}
      \item {\color{black} Every point in the $\omega$-limit set $\Omega(y_0)$ of $y(t,y_0)$ is an equilibrium of~\eqref{eq:Het-SIRI}. 
      \item $y(t,y_0)$ converges to a point in $\mathcal{M}$ as $t \to \infty$ if $R_1 \leq 1$.
      \item $y(t,y_0)$ converges to a point in $\mathcal{M}$ or to the EE as $t \to \infty$ if $R_1 > 1$.}
  \end{itemize}
\end{lemma}
\begin{proof}
By invariance of $\Delta_N$, any solution $y = y(t,y_0)$ of~\eqref{eq:Het-SIRI} with initial condition $y_0 \in \Delta_N$ is bounded and stays in $\Delta_N$ for $t\geq 0$. Therefore its $\omega$-limit set $\Omega(y_0)$ is a nonempty, compact, invariant set, and $y$ approaches $\Omega(y_0)$ as $t \to \infty$ (see Lemma 4.1 in~\cite{khalil2002nonlinear}). Let $V = \mathbf{1}^T p^S$, then $\dot{V} = - (p^S)^T B p^I \leq 0$ in $\Delta_N$. By LaSalle's Invariance Principle~\cite{khalil2002nonlinear}, $y$ approaches the largest invariant set $W$ in the set $L =\{(p^S,p^I) \in \Delta_N \, | \, \dot{V}=0\}$. It follows that, on $L$, $P^{S} B p^{I} = \mathbf{0}$ which implies $\dot{p}^S = \mathbf{0}$. Moreover, since $\hat{B}$ has a zero at every entry where $B$ has a zero, it follows that $P^S \hat{B} p^I=\mathbf{0}$ on $L$.
This in turn implies that the $p^I$ dynamics on $L$ are given by the network SIS dynamics~\eqref{eq:SIS}.
Since solutions of~\eqref{eq:SIS} either converge to the IFE point $p^{I*} = \mathbf{0}$ or to the EE~\eqref{eq:EE_j} as $t \to \infty$~\cite{fall2007epidemiological}, it follows that every invariant set of $L$ consists only of equilibria of~\eqref{eq:Het-SIRI} (see Remark~\ref{remark:SIRI_SIS_equivalence}). By Proposition~\ref{prop:equilibria}, $W$ is the union of the IFE set $\mathcal{M}$ and the EE~\eqref{eq:EE_j}. Furthermore, since all $\omega$-limit points are equilibria, $\Omega(y_0)$ contains a single point corresponding to either a point in the IFE set $\mathcal{M}$ or the EE. 
{\color{black}If $R_1 \leq 1$ it follows from Propositions~\ref{prop:equilibria} and~\ref{prop:EEstrong} that the IFE set $\mathcal{M}$ comprises the only equilibria of~\eqref{eq:Het-SIRI}. Therefore, $y$ converges to a point in $\mathcal{M}$ as $t \to \infty$. If $R_1 > 1$, $y$ converges to a point in $\mathcal{M}$ or the EE as $t \to \infty$.}
\end{proof}

\subsection{Behavioral Regimes}
{\color{black}We now state the third theorem of the paper, which shows how the four reproduction numbers distinguish the four behavioral regimes: infection-free, endemic, epidemic, and bistable}. We interpret and illustrate in Figure~\ref{fig:NSIRI_Quad}. The proof follows.

\begin{theorem}[Behavioral Regimes]\label{thm:NSIRImain} Let $p^I(0) \succ \mathbf{0}$, and $w^T_m$ be the leading left-eigenvector {\color{black} associated with $R_{max}$}. Then the network SIRI model~\eqref{eq:Het-SIRI} exhibits four qualitatively distinct behavioral regimes:
\begin{enumerate}
\item Infection-Free Regime: If $R_{max} \leq 1$  the following hold:
\begin{enumerate}
    \item All solutions converge to a point in $\mathcal{M}$ as $t \to \infty$.
    \item If $B \succeq \hat{B}$ or $\hat{B} \succeq B$,  the weighted  infected average $p^I_{\rm avg}(t) = w^T_m D^{-1} p^I(t)$ decays monotonically to zero.
\end{enumerate}\label{itm:IFE_Reg}

\item Endemic Regime: If $R_{min} > 1$,  all solutions converge to the EE as $t \to \infty$.\label{itm:EE_Reg}

\item Epidemic Regime: If {\color{black}$R_{max} > 1 > R_{min}$}, and $R_1 \leq 1$,  the following hold:
\begin{enumerate}
    \item All solutions converge to a point in $\mathcal{M}$ as $t \to \infty$.
    \item There exists $H \subset \Delta_N$ and $H \supset \mathcal{M}_+$ that is foliated by families of heteroclinic orbits, each orbit connecting two points in $\mathcal{M}$. 
\end{enumerate}\label{itm:Epi_Reg}

\item Bistable Regime: If {\color{black}$R_{max} > 1 > R_{min}$} and $R_1 > 1$, then, depending on the initial conditions, solutions converge to a point in $\mathcal{M}$ or  to the EE as $t \to \infty$.
 \label{itm:Bis_Reg} 
\end{enumerate}
\end{theorem}

Table~\ref{tab:NSIRI_BehavRegimes} summarizes which regimes of Theorem~\ref{thm:NSIRImain}  exist for each of the six cases of the network SIRI model. Figure~\ref{fig:NSIRI_Quad} illustrates the four regimes of Theorem~\ref{thm:NSIRImain} near the IFE set $\mathcal{M}$ when $B \succeq \hat B$ or $\hat{B} \succeq B$. 
 
 \begin{table}[!t]
\caption{Behavioral Regimes of the Network SIRI Cases}
\label{tab:NSIRI_BehavRegimes}
\begin{center}
\begin{tabular}{c|c|c|c|c|c}
\textbf{Case} & \textbf{Model} &\textbf{Inf.-Free} & \textbf{Endemic} & \textbf{Epidemic} & \textbf{Bistable}\\
\hline
1 & SI & \ & \checkmark &  & \\
2 & SIR & \checkmark & & \checkmark & \\
3 & SIS & \checkmark & \checkmark & & \\
4 & Partial & \checkmark & \checkmark & \checkmark  & \\
5 & Comprom. & \checkmark & \checkmark &  & \checkmark\\
6 & Mixed & \checkmark & \checkmark & \checkmark & \checkmark
\end{tabular}
\end{center}
\end{table}
 
 In Case 2 (SIR), since $\hat{B} = \bar{\mathbf{0}}$, $R_{max} = R_0$ and $R_{min} = R_1 = 0$. So only the infection-free and epidemic regimes are possible. This corresponds to the line $R_1 = 0$ in  Figure~\ref{fig:NSIRI_Quad}.  
 
 In Case 3 (SIS), since $B = \hat{B}$, $R_{max} = R_0 = R_1 = R_{min}$. So only the infection-free and endemic regimes are possible. This corresponds to the line $R_1 = R_0$ in  Figure~\ref{fig:NSIRI_Quad}. 
 
 In Case 4 (partial immunity), since $B \succ \hat{B}$, by Proposition~\ref{prop:RmaxRmin},
 $R_{max}= R_0$ and $R_{min}=R_1$. So only the infection-free, endemic, and epidemic regimes are possible.  This corresponds to the region  $R_1 < R_0$ in Figure~\ref{fig:NSIRI_Quad}. 

In Case 5 (compromised immunity), since $\hat{B} \succ B$, by Proposition~\ref{prop:RmaxRmin}, $R_{max}= R_1$ and $R_{min}=R_0$.  So only the infection-free, endemic, and bistable regimes are possible. This corresponds to the region $R_1 > R_0$  in Figure~\ref{fig:NSIRI_Quad}. 

In Case 6 (mixed immunity), all four regimes are possible.

The $N$-dimensional set $\mathcal{M}$ is illustrated in Figure~\ref{fig:NSIRI_Quad} as a plane ($N=2$) for ease of visualization. The blue region represents $\mathcal{M}_-$ (the set of stable points in $\mathcal{M}$) and the red region represents $\mathcal{M}_+$ (the set of unstable points in $\mathcal{M}$). Black arrows illustrate the flow of solutions near $\mathcal{M}$. Theorems~\ref{thm:stability} and~\ref{thm:uniqueMsubsets} prove which regions of $\mathcal{M}$ exist in each of the regimes. In the infection-free regime, $\mathcal{M} = \mathcal{M}_-$ and all solutions converge to the IFE set $\mathcal{M}$. In the endemic regime, $\mathcal{M} = \mathcal{M}_+$ and all solutions converge to the EE. 
In Cases 2 and 4 in the epidemic regime and in Case 5 in the bistable regime, $\mathcal{M}_-$ and $\mathcal{M}_+$ both exist and there is a unique hypersurface $\mathcal{M}_0$ shown as a black dashed line separating the two. 

In Section~\ref{sec:BisAndEpi} we study the geometry of solutions near $\mathcal{M}$ and the stable manifold (green) and unstable manifold (magenta) of $\mathcal{M}_0$ in the epidemic regime of Cases 2 and 4 and the bistable regime of Case 5.  These manifolds are included in Figure~\ref{fig:NSIRI_Quad} and help illustrate how solutions can flow.
\begin{figure}[!t]
\centering
\includegraphics[width=3.25in]{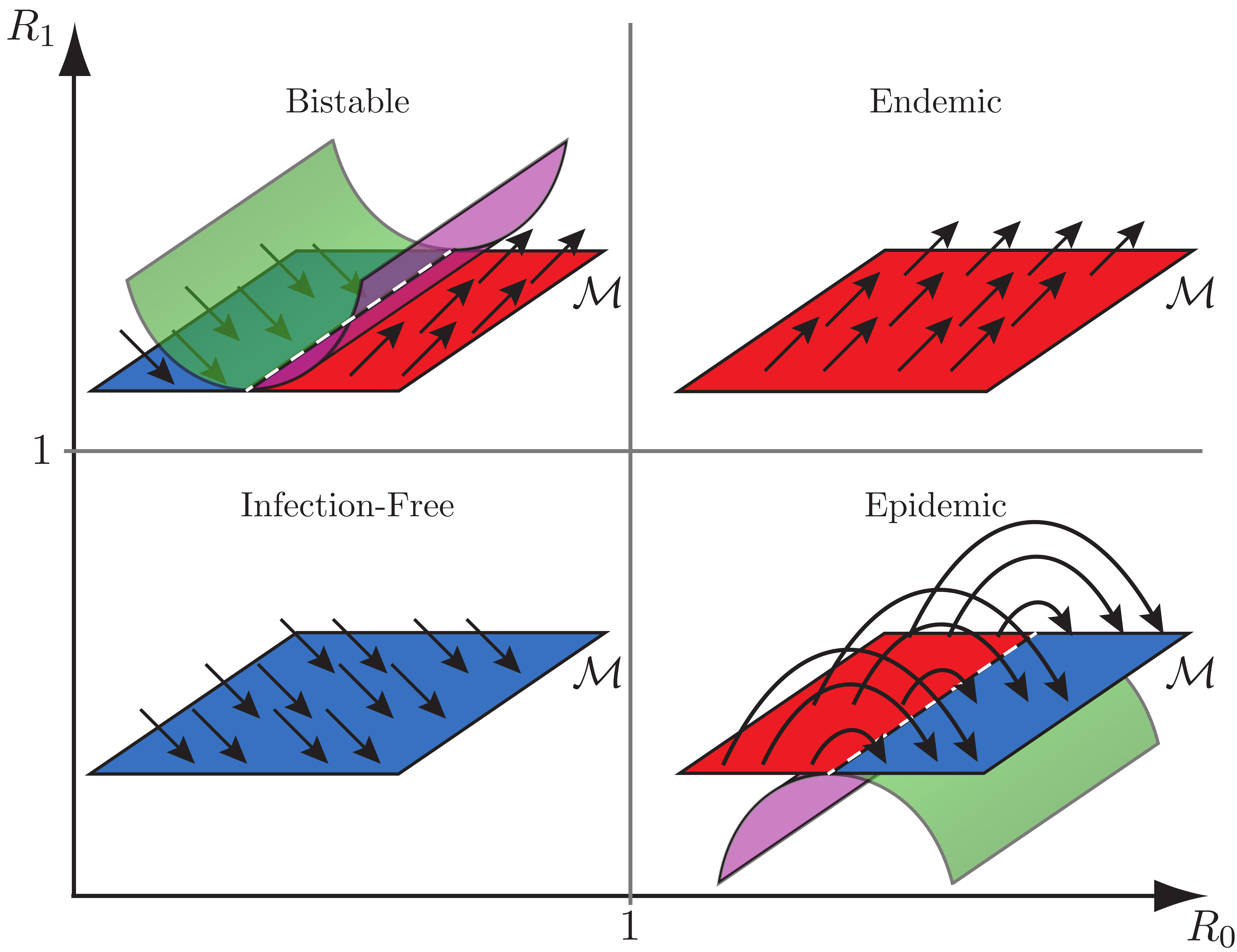}
\caption{Illustration of local dynamics near $\mathcal{M}$ for the four different behavioral regimes of the network SIRI model \eqref{eq:Het-SIRI} when $B \succeq \hat B$ or $\hat{B} \succeq B$. The  diagrams are arranged where they exist in the $R_0$, $R_1$ parameter space according to Theorem~\ref{thm:NSIRImain}. $\mathcal{M}_-$ is blue,  $\mathcal{M}_+$ is red, and $\mathcal{M}_0$ is the black dashed line. The stable and unstable manifolds of $\mathcal{M}_0$ are green and magenta, respectively.}
\label{fig:NSIRI_Quad}
\end{figure}

\begin{proof}[Proof of Theorem~\ref{thm:NSIRImain}]
To prove~\ref{itm:IFE_Reg}, let $R_{max} \leq 1$. Then by definition $R_1 \leq 1$. By Lemma~\ref{lemma:IFEorEE_as_t2infty}, all solutions converge to a point in $\mathcal{M}$ as $t \to \infty$.
By (\ref{eq:Het-SIRI}), the dynamics of  $p^I_{\rm avg}$ are
\begin{align}\label{eq:wAvgDyn}
 \dot{p}^I_{\rm avg} &= w_m^T D^{-1} (B^*(p^S)-D)p^I - w_m^T D^{-1}P^I\hat{B}p^I  \nonumber \\
&\leq w_m^T D^{-1} (\bar{B}_{max}-D)p^I - w_m^T D^{-1}P^I\hat{B}p^I \nonumber \\
&= (\bar{R}_{max}-1) w_m^T p^I - w_m^T D^{-1}P^I\hat{B}p^I.
\end{align}
The inequality follows from (\ref{eq:Bmaxmin}), and the last equality follows from $\rho(D^{-1} \bar{B}_{max}) = \rho(\bar{B}_{max} D^{-1})$. Since $\hat{B}$ is irreducible and by Proposition~\ref{prop:Metzler} $w_m \gg \mathbf{0}$, the nonlinear term $w_m^T D^{-1}P^I\hat{B}p^I$ is nonnegative. And since $B \succeq \hat{B}$ or $\hat{B} \succeq B$, by Proposition~\ref{prop:RmaxRmin}, $\bar{R}_{max} = R_{max}$. Therefore, if $R_{max}<1$, then $\dot{p}^I_{\rm avg} < 0$ and $ p^I_{\rm avg}$ decays monotonically to zero as $t \to \infty$. 

If $R_{max} =1$, $ \dot{p}^I_{\rm avg} \leq 0$ with equality holding only at points in  $\Sigma = \{\mathbf{0} \preceq p^I \preceq \mathbf{1} \,|\, p^I_j > 0 {\rm~implies~} p_k^I=0, k \in \mathcal{N}_j \}$.

At any point in $\Sigma$, the dynamics of node $j$ where $p^I_j > 0$ reduce to $\dot{p}^I_j = -\delta_j p_j^I$. Thus, no solution can stay in $\Sigma$ except for the trivial solution $p^I = \mathbf{0}$. By LaSalle's Invariance Principle,  $p^I_{\rm avg}$ decays monotonically to zero as $t \to \infty$.

To prove~\ref{itm:EE_Reg}, let $R_{min} > 1$. Then by definition $R_1 > 1$. By Theorem~\ref{thm:stability}, all points in $\mathcal{M}$ are unstable. Therefore, no non-trivial solution can converge to a point in $\mathcal{M}$. By Lemma~\ref{lemma:IFEorEE_as_t2infty}, it follows that the $\omega$-limit set of all solutions with $p^I(0) \succ \mathbf{0}$ is the EE. Therefore, all solutions converge to the EE as $t \to \infty$. 

To prove~\ref{itm:Epi_Reg}, let {\color{black}$R_{max} > 1 > R_{min}$}, and $R_1 \leq 1$. By Theorem~\ref{thm:stability}, $\mathcal{M}_-, \mathcal{M}_+, \mathcal{M}_0 \neq \emptyset$.   By Lemma~\ref{lemma:IFEorEE_as_t2infty}, all solutions converge to a point in $\mathcal{M}$ as $t \to \infty$. By the proof of Lemma~\ref{lemma:M}, the unstable manifold of any unstable point $x \in \mathcal{M}_+$, lies partially or entirely in $\Delta_N$. 
Let $y = y (t,y_0)$ be a solution of~\eqref{eq:Het-SIRI} with initial condition $y_0$ on the unstable manifold of $x$. Then, $y$ converges to $x$ as $t \to -\infty$. By Lemma~\ref{lemma:IFEorEE_as_t2infty}, $y$ converges to a  point $x' \in \mathcal{M}$, $x' \neq x$, as $t \to \infty$. Thus, $y$ forms a heteroclinic orbit. Let $H \subset \Delta_N$ be the union of the unstable manifolds of all points $x \in \mathcal{M}_+$. Then, every solution in $H$ forms a heteroclinic orbit connecting two points in $\mathcal{M}$.

To prove~\ref{itm:Bis_Reg}, let $R_1 > 1$. By Proposition~\ref{prop:EEstrong}, the EE exists and it is locally stable. Therefore any solution in the region of attraction of the EE converges to the EE as $t \to \infty$. By Lemma~\ref{lemma:M}, for any locally stable point $x \in \mathcal{M}_-$, there exists $V \subset \Delta_N$ and $x \in V$ such that any solution starting in $V$ converges to a point in $V \cap \mathcal{M}_-$ at an exponential rate.
\end{proof}

\section{Bistable and Epidemic Regimes}\label{sec:BisAndEpi}

\subsection{Geometry of Solutions near $\mathcal{M}$}
In this section we examine the geometry of solutions near the IFE set $\mathcal{M}$ in the epidemic regime for Case 2 (SIR) and Case 4 (partial immunity), and the bistable regime for Case 5 (compromised immunity). 
The bistable regime for the network SIRI model, which doesn't exist for the well-studied SIS and SIR models, generalizes that proved for the well-mixed SIRI model studied in~\cite{pagliaraSIRIWM}.  
\begin{definition}[Transversal crossing of $\Lambda_c$]
Let $y(t) = (p^S(t),p^I(t)) \in \Delta_N$, $t \geq 0$, be a solution of~\eqref{eq:Het-SIRI}. We say that $y$ crosses $\Lambda_c$ {\em transversally} if $p^S$, the projection of $y$ onto $\mathcal{M}$, crosses $\Lambda_c$ transversally. This holds if there exists a time $t'>0$ and $m \in \Lambda_c$ such that $p^S(t') = m$ and $\dot{p}^S(t')^T \nabla\Lambda(m) \neq 0$.
\end{definition}

\begin{prop}[Transversal crossing direction]\label{prop:Trans1}
Let $y(t) = (p^S(t),p^I(t)) \in \Delta_N$, $t \geq 0$, be a solution of~\eqref{eq:Het-SIRI} that crosses 
$\Lambda_c$ transversally at the point $(m,\mathbf{0}) \in \mathcal{M}$ and time $t=t'$, where $c = \Lambda(m)$. If $\dot{p}^S(t')^T \nabla\Lambda(m) < 0$, then $\Lambda$ decreases as $p^S$ crosses $\Lambda_c$, and if $\dot{p}^S(t')^T \nabla\Lambda(m) > 0$, then $\Lambda$ increases as $p^S$ crosses $\Lambda_c$. Suppose $(m,\mathbf{0}) \in \mathcal{M}_0$, then if $\dot{p}^S(t')^T \nabla\Lambda(m) < 0$, $p^S$ crosses $\mathcal{M}_{0}$ from $\mathcal{M}_{+}$ to $\mathcal{M}_{-}$ and if $\dot{p}^S(t')^T \nabla\Lambda(m) > 0$, then $p^S$ crosses $\mathcal{M}_{0}$  from $\mathcal{M}_{-}$ to $\mathcal{M}_{+}$.
\end{prop}
\begin{proof}
Since $\dot{p}^S = -P^SBp^I$, then $\dot{p}^S(t')^T \nabla\Lambda(m)$ is the derivative of $\Lambda$ at $m$ along solutions of~\eqref{eq:Het-SIRI} (see Proposition~\ref{prop:gradSys}). If $\dot{p}^S(t')^T \nabla\Lambda(m) < 0$, $\Lambda$ decreases as $p^S(t)$ crosses $\Lambda_c$ and if $\dot{p}^S(t')^T \nabla\Lambda(m) > 0$, $\Lambda$ increases as $p^S(t)$ crosses $\Lambda_c$. Suppose $(m,\mathbf{0}) \in \mathcal{M}_0$.
By Lemma~\ref{lemma:MSubsetsAndLambda}, if $\dot{p}^S(t')^T \nabla\Lambda(m) < 0$, $p^S$ crosses $\mathcal{M}_{0}$ from $\mathcal{M}_{+}$ to $\mathcal{M}_{-}$. Similarly, if $\dot{p}^S(t')^T \nabla\Lambda(m) > 0$, $p^S$ crosses $\mathcal{M}_{0}$ from $\mathcal{M}_{-}$ to $\mathcal{M}_{+}$.
\end{proof}

\begin{theorem}[Transversality of solutions]\label{thm:TransversalityCases4and5}
Consider Cases 2 and 4 in the epidemic regime ($R_0>1$, $R_1<1$) and Case 5 in the bistable regime ($R_0<1$, $R_1>1$). Assume no stubborn agents. Let  $y(t) = (p^S(t),p^I(t)) \in \Delta_N$, $t \geq 0$, be a solution of~\eqref{eq:Het-SIRI}   for which there exists a time $t'>0$ and $(m,\mathbf{0}) \in \textrm{int}(\mathcal{M})$, such that $p^S(t') = m$ and $p^I(t') \succ \mathbf{0}$. Let $c = \Lambda(m)$ so that $m \in \Lambda_c$. Then $y$ crosses $\Lambda_c$ transversally. Suppose $(m,\mathbf{0}) \in \textrm{int}(\mathcal{M}_0)$. In the epidemic regime of Cases 2 and 4, $p^S$ crosses $\mathcal{M}_{0}$ from $\mathcal{M}_{+}$ to $\mathcal{M}_{-}$, and the stable and unstable manifolds of $\mathcal{M}_0$ lie outside $\Delta_N$.
In the bistable regime of Case 5, $p^S$ crosses $\mathcal{M}_{0}$ from $\mathcal{M}_{-}$ to $\mathcal{M}_{+}$, and
the stable and unstable manifolds of $\mathcal{M}_0$ lie inside $\Delta_N$.
\end{theorem}
\begin{proof}
Assume no stubborn agents. Let $t'>0$ such that
$p^S(t') = m$, $p^I(t') \succ \mathbf{0}$ and $(m,\mathbf{0}) \in \textrm{int}(\mathcal{M}$). By Definition~\ref{def:MboundaryAndInt},  $m \gg \mathbf{0}$.  Let $c = \Lambda(m)$. So $\dot{p}^S(t') = -\textrm{diag}(m)Bp^I(t') \prec \mathbf{0}$.
By Lemma~\ref{lemma:LambdaDyn}, if $B \succ \hat{B}$ (Cases 2 and 4) then $\nabla\Lambda \gg \mathbf{0}$ and if $\hat{B} \succ B$ (Case 5) then $\nabla\Lambda \ll \mathbf{0}$. Thus $\dot{p}^S(t')^T \nabla\Lambda(m) \neq 0$, and so, by Proposition~\ref{prop:Trans1}, $y$ crosses $\Lambda_c$ transversally.  Suppose $(m,\mathbf{0}) \in \textrm{int}(\mathcal{M}_0$). Then, by Proposition~\ref{prop:Trans1}, $p^S$ crosses from $\mathcal{M}_{+}$ to $\mathcal{M}_{-}$ in Cases 2 and 4, and from $\mathcal{M}_{-}$ to $\mathcal{M}_{+}$ in Case 5. By continuity of solutions with respect to initial conditions, it follows that the stable and unstable manifolds of $\mathcal{M}_0$ must lie outside $\Delta_N$ in the epidemic regime of Cases 2 and 4 and inside $\Delta_N$ in the bistable regime of Case 5, as illustrated in Figure~\ref{fig:NSIRI_Quad}.
\end{proof}

\begin{corollary}\label{corollary:Cases2and4}
In the epidemic regime of Cases 2 and 4, every heteroclinic orbit in $\Delta_N$ connects a point in $\mathcal{M}_+$ to a point in $\mathcal{M}_-$.
\end{corollary}
\begin{proof}
Since by Theorem~\ref{thm:TransversalityCases4and5} the stable manifold of any point in $\mathcal{M}_0$
lies outside $\Delta_N$, it follows by Theorem~\ref{thm:NSIRImain} that the orbit connects a point in $\mathcal{M}_+$ to a point in $\mathcal{M}_-$.
\end{proof}

\begin{corollary}\label{corollary:Cases5}
Consider Case 5 in the bistable regime. Let $y(t) = (p^S(t),p^I(t))$ be a solution of~\eqref{eq:Het-SIRI}. Then it holds that
\begin{itemize}
    \item If $y$ crosses $\mathcal{M}_0$ transversally or $(p^S(0),\mathbf{0}) \in \mathcal{M}_+$,
    then $y$ converges to the EE as $t \to \infty$.  Moreover, the EE lies on the unstable manifold of $\mathcal{M}_0$.
    \item  
    The stable manifold of $\mathcal{M}_0$ intersects the boundary of $\Delta_N$ where $p^S = \mathbf{1} - p^I$.
\end{itemize}
\end{corollary}
\begin{proof}
If $y$ crosses $\mathcal{M}_0$ transversally at $t=t'$, then, by Theorem~\ref{thm:TransversalityCases4and5}, $p^S$ crosses every level surface $\Lambda_c$ in $\mathcal{M}$ transversally and crosses $\mathcal{M}_0$ from $\mathcal{M}_-$ to $\mathcal{M}_+$. By Proposition~\ref{prop:Trans1}, $\Lambda$ strictly increases along $p^S$. It follows by Lemma~\ref{lemma:MSubsetsAndLambda} that $(p^S,\mathbf{0}) \in \mathcal{M}_+$ for all $t>t'$. Since by Lemma~\ref{lemma:M} $y$ cannot converge to a point in the IFE subset $\mathcal{M}_+$, it follows by Lemma~\ref{lemma:IFEorEE_as_t2infty} that $y$ converges to the EE as $t \to \infty$. The same argument holds if $(p^S(0),\mathbf{0}) \in \mathcal{M}_+$.  It follows that the EE lies on the unstable manifold of $\mathcal{M}_0$.

Consider any point $y = (p^S,p^I) \in \Delta_N$ on the stable manifold of $\mathcal{M}_0$. Then, as $t \to -\infty$ and  $y(t) \in \Delta_N$, $y$ remains on the stable manifold of $\mathcal{M}_0$ and $(p^S,\mathbf{0})$ remains in $\mathcal{M}_-$. By Lemma~\ref{lemma:M}, points in $\mathcal{M}_-$ have no unstable manifold in $\Delta_N$ so the stable manifold of $\mathcal{M}_0$ cannot intersect $\mathcal{M}_-$. Instead, because the components of $p^S$ increase monotonically as $t \to -\infty$, $y$ intersects $\partial\Delta_N$ where $p^S = \mathbf{1}-p^I$.
\end{proof}

The locations of the stable and unstable manifolds of $\mathcal{M}_0$, as proved in Theorem~\ref{thm:TransversalityCases4and5}, are illustrated in Figure~\ref{fig:NSIRI_Quad}:  outside $\Delta_N$ in the epidemic regime and inside $\Delta_N$ in the bistable regime.  The figure shows the heteroclinic orbits proved in Corollary~\ref{corollary:Cases2and4} for the epidemic regime.  The solutions along the heteroclinic orbits cross $\mathcal{M}_0$ transversally, with their projection onto $\mathcal{M}$ crossing from $\mathcal{M}_+$ to $\mathcal{M}_-$, as proved in Theorem~\ref{thm:TransversalityCases4and5}.

Figure~\ref{fig:NSIRI_Quad} also shows the local flow in the bistable regime, as proved in Corollary~\ref{corollary:Cases5}. These solutions also cross $\mathcal{M}_0$ transversally, but with their projection onto $\mathcal{M}$ crossing from $\mathcal{M}_-$ to $\mathcal{M}_+$, as proved in Theorem~\ref{thm:TransversalityCases4and5}.   


\subsection{Bistability and Resurgent Epidemic}
In this section we examine the bistable regime of Theorem~\ref{thm:NSIRImain}, which exists for Case 5 (compromised immunity) and for Case 6 (mixed immunity). We show how solutions in this regime can exhibit a {\em resurgent epidemic} in which an initial infection appears to die out for an arbitrarily long period of time, but then abruptly and surprisingly resurges to the EE. 

Conditions on the initial state that predict a resurgent epidemic were proved for the well-mixed SIRI model in~\cite{pagliaraSIRIWM}. {\color{black} Here we compute a critical condition on the initial state in the special case that $\mathcal{G}$ is a $d$-regular digraph, i.e., every agent $j$ in $\mathcal{G}$ has degree $d_j=d$,} and every node has the same initial state. We then illustrate numerically for a more general digraph with compromised immunity  in Figure~\ref{fig:bistability}.

\begin{figure}[!t]
\centering
\includegraphics[width=3.5in]{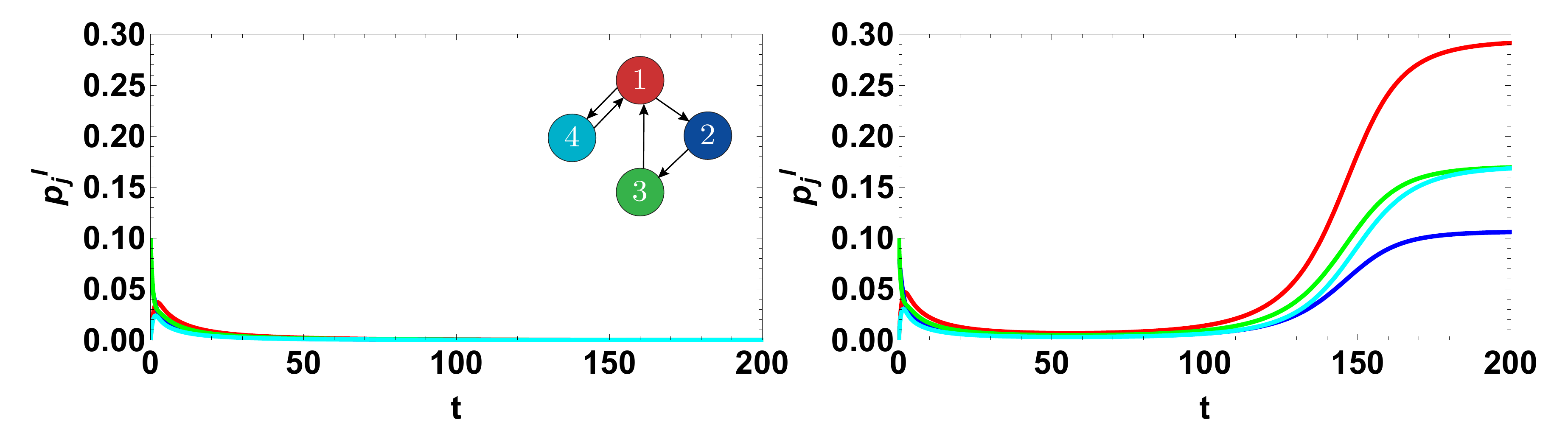}
\caption{Bistability and resurgent epidemic. Simulation of $p^I_j$ versus time $t$  for network of $N=4$ agents in  bistable regime of Case 5 (compromised immunity) with $p^S(0) = \mathbf{1}-p^I(0)$:  $j=1$ in red, $j=2$ in blue, $j=3$ in green, and $j=4$ in cyan. $A$ is the unweighted adjacency matrix of the digraph shown. $B = 0.7 A$, $\hat{B} = \textrm{diag}([1.5,0.7,0.7,0.7])A$, and $D = \mathbb{I}$. Left. $p^I(0) = [0,0.05,0.1,0]^T$. Right. $p^I(0) = [0,0.08,0.1,0]^T$. 
}
\label{fig:bistability}
\end{figure}

Consider a $d$-regular digraph with global recovery, infection, and reinfection rates: $D = \delta \mathbb{I}$, $B = \beta A$, and $\hat{B} = \hat{\beta} A$, where $A$ is the adjacency matrix. 
The network SIRI  dynamics~\eqref{eq:pj} are
\begin{align}
\dot{p}^S_j &= -\beta d p^S_j p^I_j - \beta p^S_j \sum_{k=1}^N a_{jk} (p^I_k - p^I_j) \nonumber \\
\dot{p}^I_j &=  -\delta p^I_j + (\beta-\hat{\beta}) d p^S_j p^I_j + \hat{\beta} d p^I_j - \hat{\beta}d (p^I_j)^2 \nonumber \\
&+ \left((\beta-\hat{\beta}) p^S_j + \hat{\beta} (1-p^I_j) \right) \sum_{k=1}^N a_{jk} (p^I_k - p^I_j),
\label{eq:globalSIRI}
\end{align}
where we have used the identity
$
\sum_{k=1}^N a_{jk} p^I_k = p^I_j d + \sum_{k=1}^N a_{jk} (p^I_k-p^I_j)$. Let $p^I(0) = p_{ic} \mathbf{1}$. Then (\ref{eq:globalSIRI}) reduce to
\begin{align}\label{eq:reg_graph}
\dot{p}^S_j &= -\beta d p^S_j p^I_j \nonumber\\
\dot{p}^I_j &=  -\delta p^I_j + (\beta-\hat{\beta}) d p^S_j p^I_j + \hat{\beta} d p^I_j - \hat{\beta}d(p_j^I)^2. 
\end{align}
\eqref{eq:reg_graph} describes identical and uncoupled dynamics for every agent $j$, which are equivalent to the dynamics of the well-mixed SIRI model~\cite{pagliaraSIRIWM} with infection rate $\beta d$ and reinfection rate $\hat{\beta}d$. Following~\cite{pagliaraSIRIWM}, we find the critical initial condition $p_{crit} = 1 - \xi(R_0 d \xi)^{-\beta/\hat{\beta}}$, where $\xi = (R_1-1/d)/(R_1-R_0)$. If $p_{ic} < p_{crit}$ solutions converge to a point in the IFE as $t \to \infty$.  If $p_{ic} > p_{crit}$ solutions converge to the EE, $p^{I*} = (1 - \delta/(\hat{\beta}d))\mathbf{1}$, as $t \to \infty$.
If $p_{ic} = p_{crit}$, the solution flows along the stable manifold of the point $\xi \mathbf{1} \in \mathcal{M}_0$ and converges to $\xi \mathbf{1}$. These results suggest more generally that the stable manifold of $\mathcal{M}_0$ separates solutions that converge to the IFE from those that converge to the EE, as in~\cite{pagliaraSIRIWM}.

Further, as in~\cite{pagliaraSIRIWM}, if $p_{ic} > p_{crit}$, the solution exhibits a {\em resurgent epidemic} in which $p^I_j$ initially decreases to a minimum value $p^I_{min}$ and then increases to the EE. As $(p_{ic} - p_{crit}) \to 0$,  $p^I_{min} \to 0$  and the time it takes for the solution to resurge goes to infinity. That is, the infection may look like it is gone for a long time before resurging without warning.
We note that the SIR and SIS models do not admit a bistable regime and therefore fail to account for the possibility of a resurgent epidemic.  This implies the SIR and SIS models are not robust to variability in infection and reinfection rates. 

Figure~\ref{fig:bistability} illustrates the bistability and resurgent epidemic phenomena for an example network of $N=4$ agents for Case 5 (compromised immunity). $B = 0.7 A$, $\hat{B} = \textrm{diag}([1.5,0.7,0.7,0.7])A$, and $D = \mathbb{I}.$  That is, agents $2,3,4$ are stubborn and agent 1 acquires compromised immunity to all its infected neighbors (agents 2 and 4). So $R_{max} = R_1 = 1.28$ and $R_{min} = R_0 = 0.85$, placing the system in the bistable regime.
In both panels of Figure~\ref{fig:bistability}, $p^S(0) = \mathbf{1}-p^I(0)$. In the left panel $p^I(0) = [0,0.05,0.1,0]^T$ and the solution can be observed to converge to the IFE, i.e., $p^I_j \to 0$ for all $j$. In the right panel $p^I(0) = [0,0.08,0.1,0]^T$ and there is a resurgent epidemic: each $p^I_j$ initially decays and then remains close to zero until $t \approx 100$, after which the $p^I_j$ increase rapidly to the EE, which is $p^{I*} = [0.29,0.11,0.17,0.17]^T$. 

\section{Control Strategies}\label{sec:Control}
We apply our theory to design control strategies for technological, as well as biological and behavioral settings, that guarantee desired steady-state behavior, such as the eradication of an infection. We begin with an example network with mixed immunity in the endemic regime, which has an infected steady state. We show three strategies for changing parameters that modify the reproduction numbers $R_0$, $R_1$, $R_{min}$, and $R_{max}$, and control the dynamics to a behavioral regime that results in an infection-free steady state, according to Theorem~\ref{thm:NSIRImain}.  We then consider two example networks with mixed immunity in the bistable regime.  We illustrate how vaccination of well-chosen agents increases the set of initial conditions that yield an infection-free steady state or at least delay a resurgent epidemic so that further control can be introduced.

\subsection{Control from Endemic to Infection-Free Steady State}
Consider the network of four agents shown in the top left panel of Figure~\ref{fig:simulations}.  Let $B=A$, the unweighted adjacency matrix for the digraph shown. Let $\hat{B} = \textrm{diag}([0.3,1,2,1])A$ and $D = \mathbb{I}$.  The network has weak mixed immunity (Case 6a): agent 1 acquires partial immunity to reinfection, agent 3 acquires compromised immunity to reinfection, and agents 2 and 4 are stubborn.  We compute the reproduction numbers:  $R_0=1.32$, $R_1=1.22$, $R_{min}=1.13$, and $R_{max}=1.52$, which{\color{black}, by Theorem~\ref{thm:NSIRImain},} imply dynamics in the endemic regime and an infected steady state for every initial condition.  This is illustrated in the simulation of $p^I_j$ versus time $t$ for initial conditions $p^S(0) = \mathbf{1}-p^I(0)$, $p^I(0) = [0.01,0.01,0.01,0.2]^T$.  By~\eqref{eq:EE_j}, the solution converges to the EE: $p^{I*} = [0.07,0.23,0.12,0.19]^T$. 
\begin{figure}[!t]
\centering
\includegraphics[width=3.5 in]{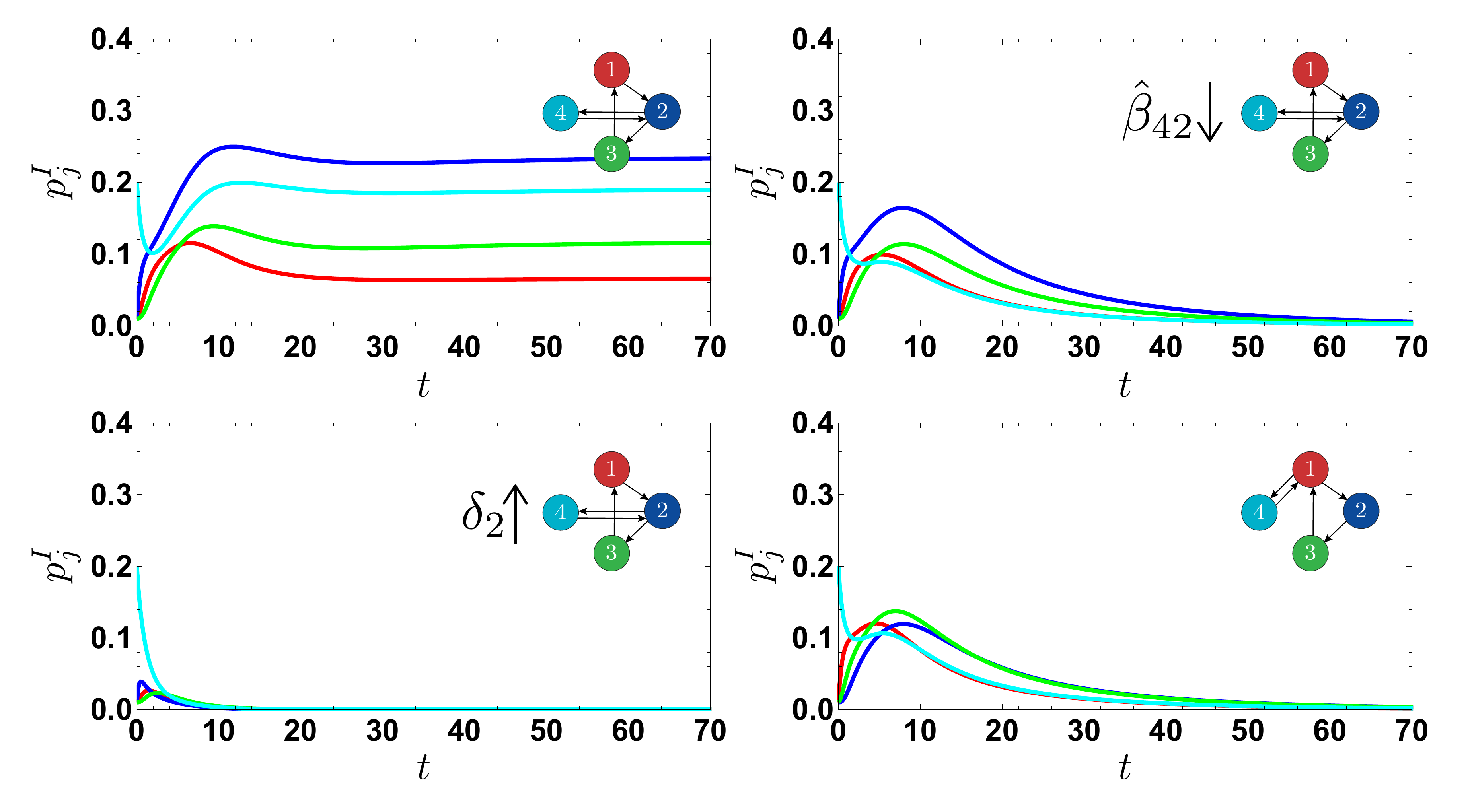}
\caption{Simulations of $p^I_j$ vs.\ $t$ to illustrate control strategies that eradicate infection: $j=1$ in red, $j=2$ in blue, $j=3$ in green, and $j=4$ in cyan. Top Left. Example network of 4 agents with weak mixed immunity (Case 6a) in the endemic regime. $B = A$, $\hat{B} = \textrm{diag}([0.3,1,2,1])A$, and $D = \mathbb{I}$. 
Bottom Left. Modification of recovery rate of agent 2 from $\delta_2 =1$ to $\delta_2 = 3.5$.  
Top Right. Modification of reinfection rate  $\hat{\beta}_{42} = 1$ to $\hat{\beta}_{42} = 0.3$. 
Bottom Right. Modification of network topology as shown.
}
\label{fig:simulations}
\end{figure}

\subsubsection{Modification of agent recovery rate}
This strategy controls the network behavior by selecting one or more agents for treatment to increase its recovery rate.  In the epidemiological setting, this could mean medication.  In the behavioral setting, this could mean providing incentives or training.  We make just one modification to the example network of four agents:  $\delta_2 = 1$ becomes $\delta_2 = 3.5$. The corresponding reproduction numbers are $R_0=0.80$, $R_1=0.72$, $R_{min}=0.65$, and $R_{max}=0.94$, which{\color{black}, by Theorem~\ref{thm:NSIRImain},} imply dynamics in the infection-free regime. Using the same initial conditions as in the top left panel, we simulate the modified system in the bottom left panel of Figure~\ref{fig:simulations}.  The  solution converges to an infection-free steady state as predicted by Theorem~\ref{thm:NSIRImain}.

\subsubsection{Modification of agent reinfection rate}
This strategy controls the network behavior by selecting one or more agents for treatment to decrease its reinfection rate.  In the epidemiological setting, this is vaccination, and in the behavioral setting, it could be inoculation as in psychology research~\cite{roozenbeek2019}.
We make just one modification to the example network of four agents:  $\hat{\beta}_{42} = 1$ becomes $\hat{\beta}_{42} = 0.3$. The corresponding reproduction numbers are $R_0=1.32$, $R_1=0.96$, $R_{min}=0.82$, and $R_{max}=1.52$, which{\color{black}, by Theorem~\ref{thm:NSIRImain},} imply dynamics in the epidemic regime. Using the same initial conditions as in the top left panel, we simulate the modified system in the top right panel of Figure~\ref{fig:simulations}.  After a small and short-lived epidemic, the solution converges to an infection-free steady state,  as predicted by Theorem~\ref{thm:NSIRImain}.

\subsubsection{Modification of network topology}
This strategy controls the network behavior by selecting one or more edges in the network graph for re-wiring.  In all  settings, this means affecting who comes in contact with whom.  
For the example network, we move the connections between agents 4 and 1 to  be 
between agents 4 and 2. The corresponding reproduction numbers 
are identical to those in the modification of reinfection rate example and therefore{\color{black}, by Theorem~\ref{thm:NSIRImain},} also imply dynamics in the epidemic regime. Using the same initial conditions as in the top left panel, we simulate the modified system in the bottom right panel of Figure~\ref{fig:simulations}.   The solution converges to an infection-free steady state,  as predicted by Theorem~\ref{thm:NSIRImain}, with an initial epidemic smaller than in the top right panel. 

\subsection{Control in Bistable Regime}

\subsubsection{Small network} A network of two agents with weak mixed immunity is shown in Figure~\ref{fig:MPlanes}. Agent 1 acquires compromised immunity while agent 2 acquires partial immunity: $\beta_{12} = \hat{\beta}_{21} = 0.8$, $\hat{\beta}_{12} = \beta_{21} = 1.3$, $\delta_1 = \delta_2 =1$. {\color{black}The corresponding reproduction numbers are $R_0=R_1=1.02$, $R_{min}=0.8$, and $R_{max}=1.3$. By Theorem~\ref{thm:NSIRImain}, the system is in the bistable regime}. Suppose initially there are no recovered agents, i.e., $p^S(0) = \mathbf{1} - p^I(0)$. Then, it can be shown that the solution will always converge to the EE.  Now suppose we apply a control strategy in which we vaccinate the agent who acquires partial immunity, where vaccination is equivalent to exposing the agent to the infection.
After the vaccination of agent 2 in our example, we have $p^S_1(0) = 1 - p^I_1(0)$ with $p^I_1(0) \in [0,1]$ and $p^S_2(0) = 0$ with $p^I_2(0) \in [0,1]$.  

We illustrate the results of the vaccination of agent 2 in the left panel of Figure~\ref{fig:MPlanes}. 
Initial conditions that lead to the EE are shown in magenta and to an infection-free steady state in yellow.  We illustrate with simulations using the initial conditions $p^I_1(0)$ and $p^I_2(0)$ denoted by  the star in the left panel of Figure~\ref{fig:MPlanes}. The top right simulation is of the system with no vaccination: there is an infected steady state as predicted.  The bottom right simulation is of the system with agent 2 vaccinated: there is an infection-free steady state as predicted.

\begin{figure}[!t]
\centering
\includegraphics[width=3.25in]{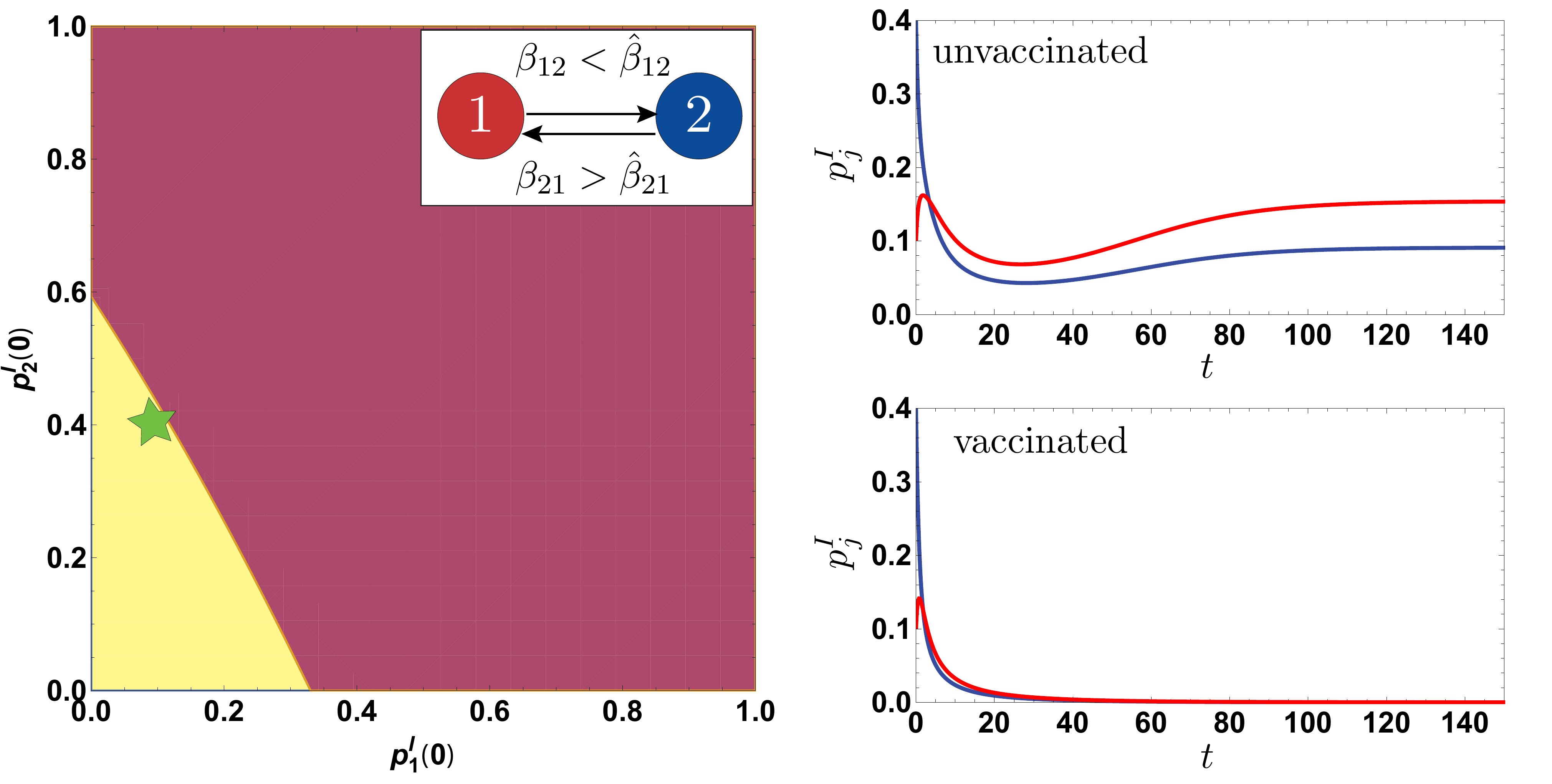}
\caption{Vaccination of agent 2 in network of two agents with mixed immunity in bistable regime. Agents 1 and 2 acquire compromised and partial immunity, respectively:
$\beta_{12} = \hat\beta_{21} = 0.8$, $\hat{\beta}_{12} = {\beta}_{21} = 1.3$, $\delta_1 = \delta_2 =1$.
Left.  
With agent 2 vaccinated, solutions with initial conditions in  magenta  converge to the EE, and 
in  yellow  converge to an IFE. 
Top Right. Simulation with no vaccination. Bottom Right. Simulation with vaccination. The initial condition is $p_1^I(0)=0.1$ and $p_2^I(0)=0.4$, shown as a star in the left panel.}
\label{fig:MPlanes}
\end{figure}

\subsubsection{Large network} A network of twenty agents with weak mixed immunity is shown in Figure~\ref{fig:LargeGraph}.  Agents 2, 3, 5, 7, 11, 13, 17, and 19 (dark gray) acquire partial immunity to reinfection: $\beta_{jk} = 0.5$, $\hat{\beta}_{jk} = 0.1$.  Agents 1 and 20 (gray) are stubborn: $\beta_{jk} = \hat{\beta}_{jk} = 0.5$. Agents 4, 6, 8, 9, 10, 12, 14, 15, 16, and 18 (light gray) acquire compromised immunity to reinfection: $\beta_{jk} = 0.5$, $\hat{\beta}_{jk} = 0.875$. $D = \mathbb{I}$. {\color{black}The corresponding reproduction numbers are $R_0=1.06$, $R_1=1.26$, $R_{min}=0.76$, and $R_{max}=1.59$. By Theorem~\ref{thm:NSIRImain}, the system is in the bistable regime}.

In Figure~\ref{fig:LargeGraph}, we plot simulations of the average infected state $\bar p^I = \mathbf{1}^T p^I/N$ versus $t$ for no vaccinations (solid black) and for different sets of vaccinated agents: agents 2, 3, 5, and 7 (dotted blue), agent 11 (dashed orange), agents 7 and 11 (point-dashed green), and agents 7, 11, and 13 (dashed violet). Vaccinating agent 11 has the strongest effect. The infection appears to be eradicated by vaccinating agents 7, 11, and 13. The other vaccination cases delay the epidemic, which provides time for treatments or other control interventions.

\begin{figure}[!t]
\centering
\includegraphics[width=3.5in]{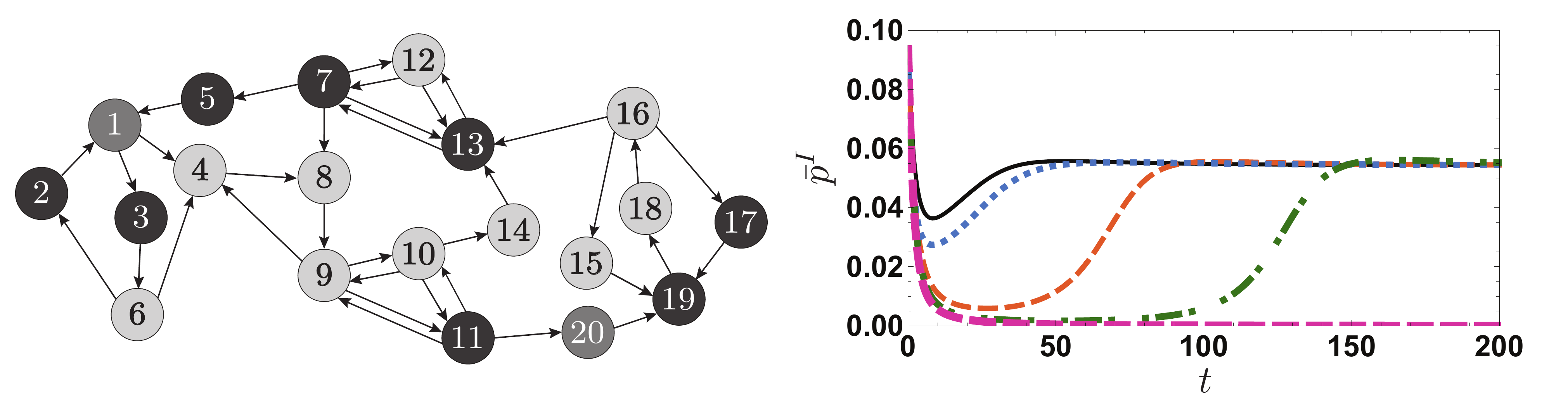}
\caption{
Left. Network of twenty agents with weak mixed immunity in bistable regime. 
Agents in dark gray acquire partial immunity to reinfection: $\beta_{jk} = 0.5$, $\hat{\beta}_{jk} = 0.1$.  Agents in gray are stubborn: $\beta_{jk} = \hat{\beta}_{jk} = 0.5$. Agents in light gray acquire compromised immunity to reinfection: $\beta_{jk} = 0.5$, $\hat{\beta}_{jk} = 0.875$. $D = \mathbb{I}$.
Right. Simulations of $\bar{p}^I$ for no vaccination (solid black), and vaccinations of agents 2,3,5,7 (dotted blue), agent 11 (dashed orange), agents 7, 11 (point-dashed green), and agents 7, 11, 13  (dashed violet). Initial conditions are $p^I_1(0) = p_{20}^I(0) = 0.5$ and $p_j^I(0)=0.05$ for $j \neq 1,20$.}
\label{fig:LargeGraph}
\end{figure}

\section{Final Remarks}\label{sec:Conclusion}

The network SIRI model generalizes the network SIS and SIR models by allowing agents to adapt their susceptibility to reinfection.
We have proved new conditions on network structure and model parameters that distinguish four behavioral regimes in the network SIRI model. The conditions depend on four scalar quantities, which generalize the basic reproduction number used in the analysis and control of the SIS and SIR models.
The SIRI model captures dynamic outcomes that the SIS and SIR models necessarily miss, most dramatic of which is the resurgent epidemic and the bistable regime.

The generality of the network SIRI model provides a means to assess robustness to uncertainty and changes in infection rates.
Further, the model provides new flexibility in control design, including, as we have illustrated, selective vaccination of agents that acquire partial immunity or modification of recovery rates, reinfection rates, and wiring, to leverage what can happen when agents adapt their susceptibility.  

Control strategies from the literature can also be extended to the SIRI setting, for example, optimal node removal, optimal link removal, and budget-constrained allocation~\cite{nowzari2016analysis,khanafer2014optimal,preciado2009spectral,nowzari2017optimal}.  However, optimal node and link removal have been shown to be NP-complete and NP-hard problems, respectively~\cite{nowzari2016analysis,van2011decreasing} and all of these strategies rely on knowledge of the network structure or system parameters.  A practical alternative is to design feedback control strategies whereby agents actively adapt their susceptibility. 

 {\color{black}Our results were obtained using the IBMF approach~\cite{van2009virus,pastor2015epidemic} to reduce the size of the state space.}
 It has been shown that if solutions in an IBMF model converge to an infection-free equilibrium, then the expected time it takes the stochastic Markov model to reach the infection-free absorbing state is sublinear with respect to network size~\cite{nowzari2016analysis}. It is an open question if bistability in the stochastic Markov model will be sensitive to noise. Other questions to be explored include deriving centrality measures that facilitate optimal control design and extending to the case in which agents adapt their susceptibility to every reinfection.
\ifCLASSOPTIONcaptionsoff
  \newpage
\fi



\bibliographystyle{IEEEtran}
\bibliography{references}

%

\begin{IEEEbiography}[{\includegraphics[clip,height=1.25in,width=1in,keepaspectratio]{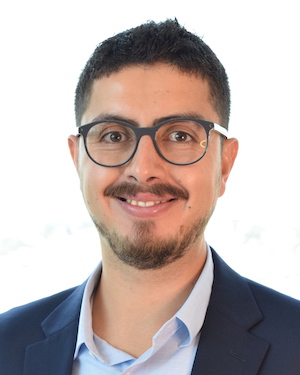}}]{Renato Pagliara}
received the B.A.Sc degree in biomedical engineering from Simon Fraser University, Burnaby, BC, Canada, in 2012, and the M.A.
and Ph.D. degrees in mechanical and aerospace engineering from Princeton University, Princeton, NJ, USA, in
2017 and 2019, respectively. 

His research interests include modeling and analysis of collective behavior in biological and engineered systems, and control for multiagent systems.
\end{IEEEbiography}

\begin{IEEEbiography}[{\includegraphics[width=1in,height=1.25in,clip,keepaspectratio]{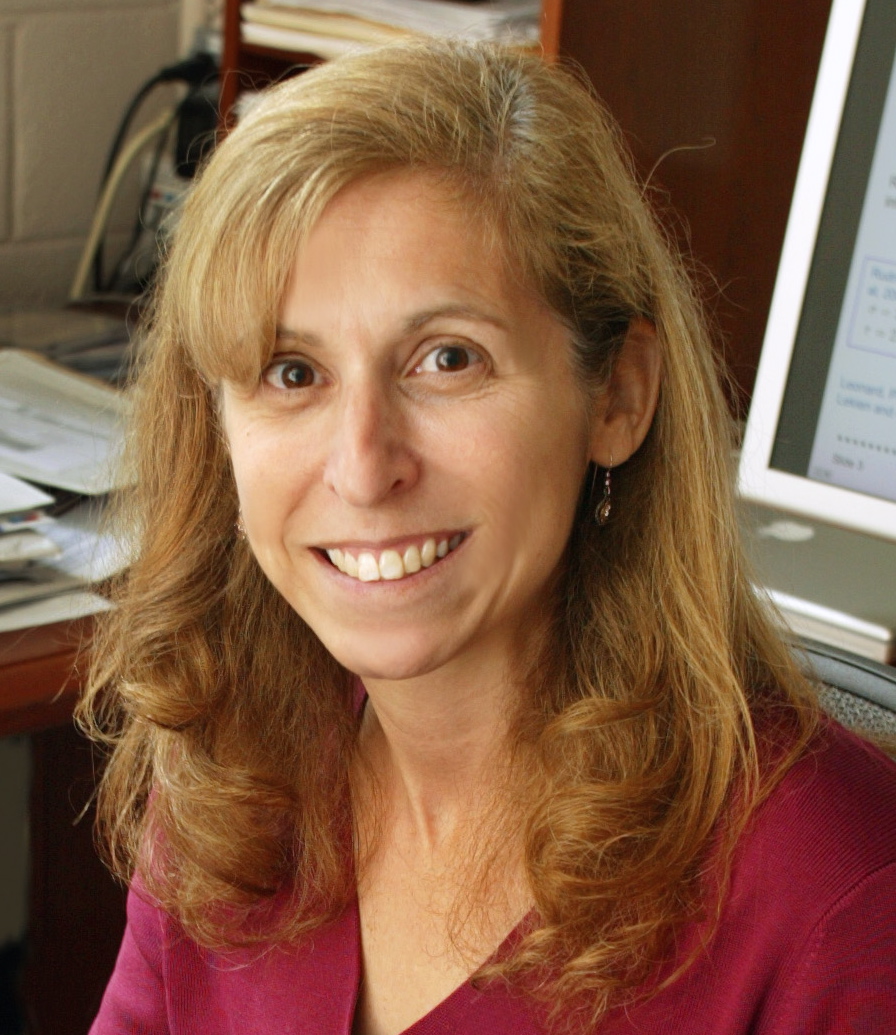}}]{Naomi Ehrich Leonard}
(F’07) received the B.S.E. degree in mechanical engineering from Princeton University, Princeton, NJ, USA, in 1985, and the M.S.
and Ph.D. degrees in electrical engineering from the University of Maryland, College Park, MD, USA, in
1991 and 1994, respectively. 

From 1985 to 1989, she was an Engineer in the electric power industry. She is currently the Edwin
S. Wilsey Professor with the Department of Mechanical and Aerospace Engineering and the Director of the Council on Science and Technology at Princeton University. She is also an Associated Faculty of Princeton's Program in Applied and Computational Mathematics. Her research and teaching are in control and dynamical systems with current interests including control for
multiagent systems, mobile sensor networks, collective animal behavior, and social decision-making dynamics.
\end{IEEEbiography}







\end{document}